\documentclass{amsart}

\usepackage{amssymb} 
\usepackage[cmtip,all]{xy} 

\usepackage{tikz}
\usetikzlibrary{arrows}
 \usetikzlibrary{snakes}
 \usetikzlibrary{shapes}

\newtheorem{theorem}{Theorem}[section]
\newtheorem{lemma}[theorem]{Lemma}

\theoremstyle{definition}
\newtheorem{definition}[theorem]{Definition}

\newtheorem{proposition}[theorem]{Proposition} 
\newtheorem{corollary}[theorem]{Corollary}

\theoremstyle{remark}
\newtheorem{remark}[theorem]{Remark}

\numberwithin{equation}{section}

\newcommand{\ten}{\circledast} 

\newcommand{\rep}[1]{%
  {%
    \tiny%
    \begin{matrix}%
      #1%
    \end{matrix}%
  }%
  }

\begin{document}

\title[Hochschild-Mitchell Cohomology]{Homological epimorphisms in functor categories and Hochschild-Mitchell cohomology}



\author{Valente Santiago Vargas}
\address{Departamento de Matem\'aticas, Facultad de Ciencias, Universidad Nacional Aut\'onoma de M\'exico,
Circuito Exterior, Ciudad Universitaria, C.P. 04510, Ciudad de M\'exico, MEXICO.}
\curraddr{}
\email{valente.santiago@ciencias.unam.mx}
\thanks{}

\author{Edgar Omar Velasco P\'aez}
\address{Departamento de Matem\'aticas, Facultad de Ciencias, Universidad Nacional Aut\'onoma de M\'exico,
Circuito Exterior, Ciudad Universitaria, C.P. 04510, Ciudad de M\'exico, MEXICO.}
\curraddr{}
\email{edgar-bkz13@ciencias.unam.mx}
\thanks{}


\date{}

\dedicatory{}

\commby{}

\begin{abstract}
In this paper we study homological epimorphisms in functor categories.
Given an ideal $\mathcal{I}$ that satisfies certain conditions in a $K$-category $\mathcal{C}$, we obtain a homological epimorphism $\Phi:\mathcal{C}\longrightarrow \mathcal{C}/\mathcal{I}$.  We investigate the relationship of the Hochschild-Mitchell cohomologies $H^{i}(\mathcal{C})$  and $H^{i}(\mathcal{C}/\mathcal{I})$ of $\mathcal{C}$ and $\mathcal{C}/\mathcal{I}$, respectively,  and we show that they can be connected by a long exact sequence. This result is a generalization of the first exact sequence obtained in \cite[Theorem 3.4 (1)]{KoenigNagase} by Koenig and Nagase. As an application of our results, we study the Hochschild-Mitchell cohomology of the triangular matrix category  $\Lambda=\left[ \begin{smallmatrix}
\mathcal{T} & 0 \\ 
M & \mathcal{U}
\end{smallmatrix}\right]$ as defined in \cite{LeOS}, we show that the Hochschild-Mitchell cohomologies $H^{i}(\Lambda)$  and $H^{i}(\mathcal{U})$ can be connected by a long exact sequence. This result extends the well-known results independently discovered  by Cibils and Michelena-Platzeck; see \cite{Cibils}  and \cite{MichelanaPlat}.  Finally, we prove that a torsion free class in a $K$-category induces a homological epimorphism and, we show that certain recollement of abelian categories can be lifted to a recollement of derived categories.


. \end{abstract}

\maketitle

\section{Introduction}\label{sec:1}
Let $A$ be a finite dimensional associative algebra with identity over an algebraically closed field $K$. The Hochschild cohomology groups $H^{i}(A,X)$ of $A$ with coefficients in a finitely generated $A$-$A$-bimodule $X$ were defined by Hochschild in 1945 in \cite{Hochschild}. When $X=A$, we usually write $H^{i}(A)$ instead of $H^{i}(A,A)$, and $H^{i}(A)$ is called the $i$-th Hochschild cohomology group of $A$.\\
On the other hand, the Hochschild-Mitchell cohomology of a $K$-linear category was defined by Mitchell in \cite{Mitchelring}. It is worth mentioning that  several authors have studied  the Hochschild-Mitchell cohomology of a $K$-category,  including, C. Cibils, E. Herscovich, E. N. Marcos, A. Solotar; 
(see, \cite{Chites}, \cite{CibilsMarcos}, \cite{CibilsRedondo}, \cite{CibiSoloRedon}, \cite{HersSolotar1}, \cite{HersSolotar2},  \cite{KoenigNagase}).\\
When studying finite dimensional algebras, certain $K$-linear categories arise. For instance, given a $K$-algebra of the form $B = KQ/I$, where $KQ$ is the path algebra associated with a finite quiver $Q$ and $I$ is an admissible ideal,  the universal Galois covering $F:\mathcal{A}\longrightarrow B$ can be constructed, where $\mathcal{A}$ is a $K$-linear category.  Sometimes, in order to obtain information about the algebra $B$,  the category $\mathcal{A}$ can be useful.  For example, C. Cibils and M. J. Redondo considered in \cite{CibilsRedondo} the Hochschild-Mitchell cohomology of a $K$-category as defined in \cite{Mitchelring}, and  they proved that given a Galois covering $F:\mathcal{A}\longrightarrow B$ defined by a group $G$, there is a spectral sequence $H^{\ast}(G,H^{\ast}(A, FM))$ that converges to $H^{\ast}(B,M)$  for any bimodule $M$, where $FM$ is the induced $A$-bimodule.\\

In this paper, we study conditions on an ideal $\mathcal{I}$ of a category $\mathcal{C}$  for which we obtain a homological epimorphism $\Phi:\mathcal{C}\longrightarrow \mathcal{C}/\mathcal{I}$.  We investigate the relationship of the Hochschild-Mitchell cohomologies $H^{i}(\mathcal{C})$  and $H^{i}(\mathcal{C}/\mathcal{I})$ of $\mathcal{C}$ and $\mathcal{C}/\mathcal{I}$, respectively,  and show that they can be connected by a long exact sequence (see Theorem \ref{episuclarga}). As an application of this result, we study the Hochschild-Mitchell cohomology of the triangular matrix category  $\Lambda=\left[ \begin{smallmatrix}
\mathcal{T} & 0 \\ 
M & \mathcal{U}
\end{smallmatrix}\right]$ as defined in \cite{LeOS}, we show that the Hochschild-Mitchell cohomologies $H^{i}(\Lambda)$  and $H^{i}(\mathcal{U})$ can be connected by a long exact sequence. This result extends Cibil's and Michelena-Platzeck's well-known result (see \cite[Theorem 1.14]{MichelanaPlat}). We also construct a long exact sequence for the one-point extension category, this result is a generalization of a well-known result of D. Happel (see Corollary \ref{longHappel}). We  prove that a torsion free class in a $K$-category induces in a very canonical way a homological epimorphism (see Proposition \ref{homotorsion}). Finally, we show that certain recollements of abelian categories induces recollement of derived categories (see Theorem \ref{recolleabel}).\\


We now  briefly describe the contents on this paper.\\
In Section 2 we recall basic concepts of functor categories, and we introduce the Hochschild-Mitchell cohomology for $K$-categories.\\
In Section 3, we study ideals $\mathcal{I}$ of a category $\mathcal{C}$ and the canonical projection $\Phi:\mathcal{C}\longrightarrow \mathcal{C}/\mathcal{I}$. We recall the notion of a strongly idempotent ideal as seen in Definition \ref{kidemcat}, and we develop the theory of homological epimorphisms in functor categories. In particular, we generalize a result due to Geigle and Lenzing in \cite{GeigleLen}, which characterizes homological epimorphisms; see Proposition \ref{caractidem2}. Furthermore,  we prove that if $\Phi:\mathcal{C}\longrightarrow \mathcal{C}/\mathcal{I}$ is a homological epimorphism then $\Phi^{op}\otimes \Phi:\mathcal{C}^{op}\otimes \mathcal{C}\longrightarrow (\mathcal{C}/\mathcal{I})^{op}\otimes \mathcal{C}/\mathcal{I}$ is a homological epimorphism; see Proposition \ref{piehomoloepi}.\\
In Section 4, we prove our main result that given an idempotent ideal $\mathcal{I}$ of $\mathcal{C}$ such that $\mathcal{I}(C,-)$ is projective in $\mathrm{Mod}(\mathcal{C})$ for all $C\in \mathcal{C}$, then $\Phi:\mathcal{C}\longrightarrow \mathcal{C}/\mathcal{I}$ is a homological epimorphism and the Hochschild-Mitchell cohomology of $\mathcal{C}$ and $\mathcal{C}/\mathcal{I}$ can be connected in a long exact sequence (see Theorem \ref{episuclarga}).\\
In Section, 5 we apply our main result to study the Hochschild-Mitchell Cohomology of triangular matrix category  $\Lambda=\left[ \begin{smallmatrix}
\mathcal{T} & 0 \\ 
M & \mathcal{U}
\end{smallmatrix}\right]$, and show that the Hochschild-Mitchell cohomologies $H^{i}(\Lambda)$  and $H^{i}(\mathcal{U})$ can be connected by a long exact sequence; (see Theorem \ref{Michelena}). We also construct a long exact sequence for the one-point extension category; this result is a generalization of a well-known result of D. Happel as seen in Corollary \ref{longHappel}.  In this section, we prove that when $\mathrm{Mod}(\mathcal{C})$ is a hereditary category, there is a bijection between TTF triples in $\mathrm{Mod}(\mathcal{C})$ and homological epimorphisms of the form  $\Phi:\mathcal{C}\longrightarrow \mathcal{C}/\mathcal{I}$ (see Proposition \ref{homottf}). We also prove that a torsion free class $\mathcal{F}$ in a $K$-category induces a homological epimorphism $\pi:\mathcal{C}\longrightarrow \mathcal{C}/\mathcal{I}_{\mathcal{F}}$ (see Proposition \ref{homotorsion}). Finally, we show that certain recollements of abelian categories induces recollement of derived categories (see Theorem \ref{recolleabel}).

\section{Preliminaries}
Throughout this paper we will consider small $K$-categories  $\mathcal{C}$ over a field $K$, which means that the class of objects of $\mathcal{C}$ forms a set, the morphisms set $\mathrm{Hom}_{\mathcal{C}}(X,Y)$ is a $K$-vector space and the composition of morphisms is $K$-bilinear.  For conciseness, we will sometimes write $\mathcal{C}(X,Y)$ instead of $\mathrm{Hom}_{\mathcal{C}}(X,Y)$. Furthermore, we refer to \cite{Mitchelring} for basic properties of $K$-categories.\\
Let $\mathcal{A}$ and $\mathcal{B}$ be $K$-categories a covariant $K$-functor is funtor $F:\mathcal{A}\rightarrow \mathcal{B}$ such that $F:\mathcal{A}(X,Y)\rightarrow \mathcal{B}(F(X),F(Y))$ is a $K$-linear transformation.
For $K$-categories $\mathcal{A}$ and  $\mathcal{B}$, we consider the category of all the covariant $K$-functors, which we denote by $\mathrm{Fun}_{K}(\mathcal{A},\mathcal{B})$. Given an arbitrary small additive category $\mathcal{C}$, the category of all additive covariant functors  $\mathrm{Fun}_{\mathbb{Z}}(\mathcal{C},\mathbf{Ab})$ is denoted by $\mathrm{Mod}(\mathcal{C})$ and is called the category of left $\mathcal{C}$-modules. When
 $\mathcal{C}$ is a $K$-category, there is an isomorphism of categories $\mathrm{Fun}_{\mathbb{Z}}(\mathcal{C},\mathbf{Ab})\simeq \mathrm{Fun}_{K}(\mathcal{C},\mathrm{Mod}(K))$ where $\mathrm{Mod}(K)$ denotes the category of $K$-vector spaces. Thus, we can identify $\mathrm{Mod}(\mathcal{C})$ with $\mathrm{Fun}_{K}(\mathcal{C},\mathrm{Mod}(K))$.  If $\mathcal{C}$ is a $K$-category, we always consider its opposite  category $\mathcal{C}^{op}$, which is also a $K$-category; and we construct the category of right $\mathcal{C}$-modules $\mathrm{Mod}(\mathcal{C}^{op}):=\mathrm{Fun}_{K}(\mathcal{C}^{op},\mathrm{Mod}(K))$. It is well-known that $\mathrm{Mod}(\mathcal{C})$ is an abelian category with enough projectives and injectives; see for example,\cite[Proposition 2.3]{MitBook} on page 99 and also page 102 in \cite{MitBook}).\\
If $\mathcal{C}$ and $\mathcal{D}$ are $K$-categories, B. Mitchell defined in \cite{Mitchelring} the $K$-category tensor product  $\mathcal{C}\otimes_{K}\mathcal{D}$ with objects that are those of $\mathcal{C}\times \mathcal{D}$, and the set of morphisms from $(C,D)$ to $(C',D')$ is the tensor product of $K$-vector spaces $\mathcal{C}(C,C')\otimes_{K}\mathcal{D}(D,D')$. The $K$-bilinear composition in $\mathcal{C}\otimes_{K} \mathcal{D}$ is given as follows: $(f_{2}\otimes g_{2})\circ (f_{1}\otimes g_{1}):=(f_{2}\circ f_{1})\otimes(g_{2}\circ g_{1})$ 
for all $f_{1}\otimes g_1\in \mathcal{C}(C,C')\otimes \mathcal{D}(D,D')$ and  $f_{2}\otimes g_2\in\mathcal{C}(C',C'')\otimes_{K} \mathcal{D}(D',D'')$.\\
Now we recall an important construction given in  \cite{Mitchelring} on p. 26 that will be used throughout this paper. Let $\mathcal{C}$ and $\mathcal{A}$ be  $K$-categories where $\mathcal{A}$ is cocomplete. The evaluation $K$-functor $E:\mathrm{Fun}_{K}(\mathcal{C},\mathcal{A})\otimes_{K}\mathcal{C}\longrightarrow \mathcal{A}$ can be extended to a $K$-functor
$$-\otimes_{\mathcal{C}}-:\mathrm{Fun}_{K}(\mathcal{C},\mathcal{A})\otimes_{K}\mathrm{Mod}(\mathcal{C}^{op})\longrightarrow \mathcal{A}.$$
By definition, we have an isomorphism $F\otimes_{\mathcal{C}}\mathcal{C}(-,X)\simeq F(X)$ for all $X\in \mathcal{C}$,
which is natural in $F$ and $X$.  Let $\mathcal{A}$ and $\mathcal{C}$ be two  $K$-categories. 
There exists an isomorphism of abelian groups (see formula (2) on page 26 in \cite{Mitchelring})
\begin{equation}
\mathrm{Hom}_{\mathcal{A}}\Big(F\otimes_{\mathcal{C}}G,A\Big)\simeq \mathrm{Hom}_{\mathrm{Mod}(\mathcal{C}^{op})}\Big(G,\mathrm{Hom}_{\mathcal{A}}(F,A)\Big),
\end{equation}
which is natural for $F\in \mathrm{Fun}_{K}(\mathcal{C},\mathcal{A})$, $G\in \mathrm{Mod}(\mathcal{C}^{op})$ and $A\in \mathcal{A}$. Here $\mathrm{Hom}_{\mathcal{A}}(F,A)$ denotes the functor $\mathrm{Hom}_{\mathcal{A}}(F,A):\mathcal{C}^{op}\rightarrow \mathrm{Mod}(K)$ with value at $C$ that is the $K$-vector space $\mathrm{Hom}_{\mathcal{A}}(F(C),A)$.\\
We have the following construction.  Let $\mathcal{A},\mathcal{C}$ and $\mathcal{D}$ be three $K$-categories; and consider $F\in \mathrm{Fun}_{K}(\mathcal{C},\mathcal{A})$ and $G\in \mathrm{Mod}(\mathcal{C}^{op}\otimes_{K}\mathcal{D})$  where $\mathcal{A}$ is complete and cocomplete. We define a functor
$$F\boxtimes_{\mathcal{C}}G:\mathcal{D}\longrightarrow \mathcal{A}$$
as follows: $(F\boxtimes_{\mathcal{C}} G)(D):=F\otimes_{\mathcal{C}} G(-,D)\quad \forall D\in \mathcal{D},$.
We recall the isomorphism given in the formula $(4)$ on page 28 in \cite{Mitchelring}.\\
For $H\in  \mathrm{Fun}_{K}(\mathcal{D},\mathcal{A})$, there exists an isomorphism of $K$-vector spaces
\begin{equation}\label{certainadjun}
\mathrm{Hom}_{\mathrm{Fun}_{K}(\mathcal{D},\mathcal{A})}(F\boxtimes_{\mathcal{C}}G,H)\simeq  \mathrm{Hom}_{\mathrm{Mod}(\mathcal{C}^{op}\otimes_{K} \mathcal{D})}\Big(G, \mathrm{Hom}_{\mathcal{A}}(F,H)\Big).
\end{equation}
Here $\mathrm{Hom}_{\mathcal{A}}(F,H)$ denotes the functor $\mathrm{Hom}_{\mathcal{A}}(F,H):\mathcal{C}^{op}\otimes_{K}\mathcal{D}\rightarrow \mathrm{Mod}(K)$ whose value at $(C,D)$ is the $K$-vector space $\mathrm{Hom}_{\mathcal{A}}(F(C),H(D))$.  Now, for convenience of the reader we recall the following Proposition.
\begin{proposition}$\textnormal{\cite[Proposition 11.6]{Mitchelring}}$ \label{prodtenproy}
Let $\mathcal{A}$, $\mathcal{B}$ and $\mathcal{C}$ be three $K$-categories.
Let $F\in \mathrm{Fun}_{K}(\mathcal{C},\mathcal{A})$ and  $G\in \mathrm{Mod}(C^{op}\otimes_{K}\mathcal{D})$ where $\mathcal{A}$ is an abelian category with coproducts. Consider the following conditions:
 \begin{enumerate}
 \item [(a)] $G$ is projective in $\mathrm{Mod}(C^{op}\otimes_{K}\mathcal{D})$ and  $F(C)$ is projective in $\mathcal{A}$ for all $C\in \mathcal{C}$;
 
 \item [(b)] For all $C\in \mathcal{C}$ the functor $G(C,-):\mathcal{D}\longrightarrow \textbf{Ab}$ is projective in $\mathrm{Mod}(\mathcal{D})$, $F$ is projective in $\mathrm{Fun}_{K}(\mathcal{C},\mathcal{A})$ and $\mathcal{A}$ is an $AB4^{\ast}$-category.
 \end{enumerate}
 If one of the above conditions hold, then $F\boxtimes_{\mathcal{C}} G$ is a projective object in $\mathrm{Fun}_{K}(\mathcal{D},\mathcal{A})$.
 \end{proposition}
 
Now, by considering the field $K$, we construct a $K$-category $\mathcal{C}_{K}$ with only one object (that is,  $\mathrm{Obj}(\mathcal{C}_{K})=\{\ast\}$).
For a given $K$-category $\mathcal{C}$, there exists two natural isomorphisms of categories
$$\Phi:\mathrm{Fun}_{K}(\mathcal{C}_{K},\mathcal{C})\rightarrow \mathcal{C},\quad \quad \Delta:\mathcal{C}_{K}^{op}\otimes_{K}\mathcal{C}\rightarrow \mathcal{C}.$$
Given $C\in \mathcal{C}$ we denote by $\overline{C}:\mathcal{C}_{K}\rightarrow \mathcal{C}$ the $K$-functor such that $\Phi(\overline{C})=C$. Given $G\in \mathrm{Mod}(\mathcal{C})$ we denote by  $\underline{G}:\mathcal{C}_{K}^{op}\otimes_{K}\mathcal{C}\rightarrow \mathrm{Mod}(K)$ the functor $\underline{G}:=G\circ \Delta$. The above construction is needed for the following Definition.

\begin{definition}\label{prodtensorconcat}
Let $\mathcal{A}$ and $\mathcal{D}$ be two $K$-categories where $\mathcal{A}$ is complete and cocomplete. Let $A\in \mathcal{A}$ and $G\in \mathrm{Mod}(\mathcal{D})$. We define
$$A\ten_{K}G:=\overline{A}\boxtimes_{\mathcal{C}_{K}}\underline{G}:\mathcal{D}\longrightarrow \mathcal{A}$$
as follows: $(A\ten_{K} G)(D):=\overline{A}\otimes_{\mathcal{C}_{K}} \underline{G}(-,D)\quad \forall D\in \mathcal{D}.$
\end{definition}

Using the notation above, we recall the following result.
\begin{corollary}$\textnormal{\cite[Proposition 11.7]{Mitchelring}}$\label{prodipordopy}
Let $\mathcal{A}$ be an abelian $K$-category with coproductos. Let $A\in \mathcal{A}$ be projective in $\mathcal{A}$  and $G\in \mathrm{Mod}(\mathcal{D})$ projective in $\mathrm{Mod}(\mathcal{D})$. Thus $A\ten_{K}G$ is projective in $\mathrm{Fun}_{K}(\mathcal{D},\mathcal{A})$.
\end{corollary}
By Corollary \ref{prodipordopy}, we get the following result.

\begin{proposition}\label{projinenvelop}
Let $\mathcal{C}$ be a $K$-category. Let $M\in \mathrm{Mod}(\mathcal{C}^{op})$ be projective in $\mathrm{Mod}(\mathcal{C}^{op})$ and $N\in \mathrm{Mod}(\mathcal{C})$ projective in $\mathrm{Mod}(\mathcal{C})$.
Therefore, $M\ten_{K}N$ is projective in $\mathrm{Mod}(\mathcal{C}^{op}\otimes_{K} \mathcal{C})$. Moreover, we have that
$(M\ten_{K}N)(C',C)\simeq M(C')\otimes_{K}N(C)$ for all $(C',C)\in \mathcal{C}^{op}\otimes_{K} \mathcal{C}$.
\end{proposition}

\subsection{Hochschild cohomology}
The following definition can be found on page 56 in \cite{Mitchelring}.
\begin{definition}
Let $\mathcal{C}$ be a $K$-category.  The $\textbf{enveloping category}$ of $\mathcal{C}$, denoted by $\mathcal{C}^{e}$, is defined as $\mathcal{C}^{e}:=\mathcal{C}^{op}\otimes_{K}\mathcal{C}.$
\end{definition}
We can see $\mathcal{C}$ as an element in $\mathrm{Mod}(\mathcal{C}^{e})$. We have the following construction given in Mitchell's article \cite{Mitchelring}. For a $K$-category $\mathcal{C}$, we consider the complex $\mathbb{S}(\mathcal{C})$ in $\mathrm{Mod
}(\mathcal{C}^{e})$ whith $n$-th term that is $0$ for $n\geq -1$ and for $n\geq 1$ it is given by 

\begin{align*}
\mathbb{S}_{n}(\mathcal{C})& :=\!\!\!\!\!\!\!\!\bigoplus_{(p_{1},\dots, p_{n+1})}\mathcal{C}(-,p_{1})\ten_{K}\Big(\Big(
\mathcal{C}(p_{1},p_{2})\otimes_{K}\cdots  \otimes_{K}\mathcal{C}(p_{n},p_{n+1})\Big)\ten_{K}\mathcal{C}(p_{n+1},-)\Big),
\end{align*}
where the coproduct ranges over all $n+1$-fold sequences of objects in $\mathcal{C}$.\\
We have that $\mathcal{C}(-,p_{1})\ten_{K}\Big(\Big(
\mathcal{C}(p_{1},p_{2})\otimes_{K}\cdots  \otimes_{K}\mathcal{C}(p_{n},p_{n+1})\Big)\ten_{K}\mathcal{C}(p_{n+1},-)\Big)\in \mathrm{Mod}(\mathcal{C}^{e})$. Now, we describe the differential $d_{n}:\mathbb{S}_{n}(\mathcal{C})\longrightarrow \mathbb{S}_{n-1}(\mathcal{C})$. 

For  an element $\alpha_{0}\otimes \alpha_{1}\otimes \cdots\otimes \alpha_{n+1} \in  \Big(\mathbb{S}_{n}(\mathcal{C})\Big)(C',C)$, we define
$$[d_{n}]_{(C',C)}\Big(\alpha_{0}\otimes \alpha_{1}\otimes \cdots\otimes \alpha_{n+1} \Big):=\sum_{i=0}^{n}(-1)^{i}\Big(\alpha_{0}\otimes \cdots\otimes (\alpha_{i+1}\circ \alpha_{i})\otimes \cdots\otimes \alpha_{n+1}\Big).$$
Since $K$ is a field, we have that $\mathbb{S}_{n}(\mathcal{C})$ is projective in $\mathrm{Mod}(\mathcal{C}^{e})$ and
so in this case $(\mathbb{S}(\mathcal{C}),d)$
is a projective resolution of $\mathcal{C}$ in $\mathrm{Mod}(\mathcal{C}^{e})$ (see page 71 in \cite{Mitchelring}). This resolution is called the $\textbf{standard resolution}$ of $\mathcal{C}$. We have the following Proposition.
\begin{proposition}\label{lemaduda}
Let $G\in \mathrm{Mod}(\mathcal{C}^{op}\otimes_{K} \mathcal{D})$ be such that $G(C,-)$ is projective in $\mathrm{Mod}(\mathcal{D})$ for all $C\in \mathcal{C}^{op}$.  Consider the standard projective resolution  $\Big(\mathbb{S}(\mathcal{C}),d\Big)$ of $\mathcal{C}$. Then  $G\boxtimes_{\mathcal{C}^{op}}\mathbb{S}(\mathcal{C})$ is a projective resolution of $G\boxtimes_{\mathcal{C}^{op}} \mathcal{C}\simeq G$ in $\mathrm{Mod}(\mathcal{C}^{op}\otimes_{K} \mathcal{D})$
\end{proposition}
\begin{proof}
See p. 71 in \cite{Mitchelring}.
\end{proof}
The following is the generalization of Hochschild Cohomology to the setting for $K$-categories.
\begin{definition}\label{HochMit}
Let $\mathcal{C}$ be a $K$-category. The $n$-$\textbf{Hochschild-Mitchell cohomology}$ group of $\mathcal{C}$ is defined as $H^{n}(\mathcal{C}):=\mathrm{Ext}_{\mathrm{Mod}(\mathcal{C}^{e})}^{n}(\mathcal{C},\mathcal{C}).$
\end{definition}

\section{Homological epimorphisms in functor categories}

A  $\textbf{two}$ $\textbf{sided}$ $\textbf{ideal}$  $\mathcal{I}(-,?)$ of $\mathcal{C}$ is a $K$-subfunctor of the two variable functor $\mathcal{C}(-,?):\mathcal{C}^{op}\otimes_{K}\mathcal{C}\rightarrow\mathrm{Mod}(K)$ such that the following conditions hold: (a) if $f\in \mathcal{I}(X,Y)$ and $g\in\mathcal C(Y,Z)$, then  $gf\in \mathcal{I}(X,Z)$; and (b)
if $f\in \mathcal{I}(X,Y)$ and $h\in\mathcal{C}(U,X)$, then  $fh\in \mathcal{I}(U,Z)$. If $\mathcal{I}$ is a two-sided ideal,  we can form the $\textbf{quotient category}$  $\mathcal{C}/\mathcal{I}$ whose objects are those of $\mathcal{C}$ and where $(\mathcal{C}/\mathcal{I})(X,Y):=\mathcal{C}(X,Y)/\mathcal{I}(X,Y)$ and the composition is induced by that of $\mathcal{C}$ (see \cite{Mitchelring}). There is a canonical projection functor $\pi:\mathcal{C}\rightarrow \mathcal{C}/\mathcal{I}$ such that $\pi(X)=X$ for all $X\in \mathcal{C}$ and  $\pi(f)=f+\mathcal{I}(X,Y):=\bar{f}$ for all $f\in \mathcal{C}(X,Y)$. We also recall that there exists a canonical isomorphism of categories $(\mathcal{C}/\mathcal{I})^{op}\simeq \mathcal{C}^{op}/\mathcal{I}^{op}$.
We construct the following two functors 
$$\mathbb{D}_{\mathcal{C}}:\mathrm{Fun}_{K}(\mathcal{C},\mathrm{Mod}(K))\longrightarrow \mathrm{Fun}_{K} (\mathcal{C}^{op},\mathrm{Mod}(K))$$
$$\mathbb{D}_{\mathcal{C}^{op}}:\mathrm{Fun}_{K}(\mathcal{C}^{op},\mathrm{Mod}(K))\longrightarrow \mathrm{Fun}_{K}(\mathcal{C},\mathrm{Mod}(K)),$$
defined as $\mathbb{D}_{\mathcal{C}}(F):=\mathrm{Hom}_{K}(-,K)\circ F$ and similarly for
$\mathbb{D}_{\mathcal{C}^{op}}$.\\
Given an ideal $\mathcal{I}$ in $\mathcal{C}$, we will consider the canonical functors $\pi_{1}:\mathcal{C}\longrightarrow \mathcal{C}/\mathcal{I}$ and $\pi_{2}:\mathcal{C}^{op}\longrightarrow \mathcal{C}^{op}/\mathcal{I}^{op}$.  It is easy to show that we have functors
$$(\pi_{1})_{\ast}:\mathrm{Fun}_{K}(\mathcal{C}/\mathcal{I},\mathrm{Mod}(K))\longrightarrow \mathrm{Fun}_{K}(\mathcal{C},\mathrm{Mod}(K))$$
$$(\pi_{2})_{\ast}:\mathrm{Fun}_{K}(\mathcal{C}^{op}/\mathcal{I}^{op},\mathrm{Mod}(K))\longrightarrow\mathrm{Fun}_{K}(\mathcal{C}^{op},\mathrm{Mod}(K)),$$
such that  $\mathbb{D}_{\mathcal{C}^{op}}\circ (\pi_{2})_{\ast}=(\pi_{1})_{\ast}\circ \mathbb{D}_{(\mathcal{C}/\mathcal{I})^{op}}$.\\
Since we are following the notation given on page 26  in  \cite{Mitchelring}, for a $K$-category $\mathcal{C}$ we  
have the functor $-\otimes_{\mathcal{C}}-:\mathrm{Mod}(\mathcal{C})\otimes_{K} \mathrm{Mod}(\mathcal{C}^{op})\longrightarrow \mathrm{Mod}(K).$\\
In the paper \cite{RSS}, however,  the authors consider the functor $-\otimes_{\mathcal{C}}-:\mathrm{Mod}(\mathcal{C}^{op})\otimes_{K} \mathrm{Mod}(\mathcal{C})\longrightarrow \mathrm{Mod}(K).$ Hence, in the following propositions we will recall some results from \cite{RSS} but by using the notation $-\otimes_{\mathcal{C}^{op}}-$ instead of $-\otimes_{\mathcal{C}}-$, which was originally used in \cite{RSS}.
Therefore,  for $N\in \mathrm{Mod}(\mathcal{C}^{op})$  we consider the functor
$N\otimes_{\mathcal{C}^{op}}-:\mathrm{Mod}(\mathcal{C})\longrightarrow \mathrm{Mod}(K)$. We denote by $\mathrm{Tor}_{i}^{\mathcal{C}^{op}}(N,-):\mathrm{Mod}(\mathcal{C})\longrightarrow \mathrm{Mod}(K)$ the $i$-th left derived functor of $N\otimes_{\mathcal{C}^{op}} -$. For $M\in \mathrm{Mod}(\mathcal{C})$ we now denote by
$\mathrm{Ext}^{i}_{\mathrm{Mod}(\mathcal{C})}(M,-):\mathrm{Mod}(\mathcal{C})\longrightarrow \mathrm{Mod}(K)$ the $i$-th  derived functor of $\mathrm{Hom}_{\mathrm{Mod}(\mathcal{C})}(M,-):\mathrm{Mod}(\mathcal{C})\longrightarrow \mathrm{Mod}(K)$.\\

We recall the construction of the following functors given in \cite[Definition 3.9]{RSS} and \cite[Definition 3.10]{RSS}.
The functor $ \frac{\mathcal{C}}{\mathcal{I}}\otimes_{\mathcal{C}^{op}}-:\mathrm{Mod}(\mathcal{C})\longrightarrow\mathrm{Mod}(\mathcal{C}/\mathcal{I}) $  is given as follows: for $M\in \mathrm{Mod}(\mathcal{C})$ we set 
$\left( \frac{\mathcal{C}}{\mathcal{I}}\otimes _{\mathcal{C}^{op}}M\right)(C):= \frac{\mathcal{C}(-,C)}{\mathcal{I}(-,C)}\otimes_{\mathcal{C}^{op}}M$  for all $C\in \mathcal{C}/\mathcal{I}$. We also define the functor $\mathcal{C}(\frac{\mathcal{C}}{\mathcal{I}},-):\mathrm{Mod}\left(\mathcal{C}\right) \longrightarrow \mathrm{Mod}\left(\mathcal{C}/\mathcal{I}\right) $ as follows: for $ M\in \mathrm{Mod}(\mathcal{C})$ we set 
$\mathcal{C}(\frac{\mathcal{C}}{\mathcal{I}},M)(C)=\mathcal{C}\left(\frac{\mathcal{C}(C,-)}{\mathcal{I}(C,-)},M\right)$ for all $C\in \mathcal{C}/\mathcal{I}$.

\begin{definition}$\textnormal{\cite[Definition 3.15]{RSS}}$
We denote by $\mathbb{EXT}^{i}_{\mathcal{C}}(\mathcal{C}/\mathcal{I},-):\mathrm{Mod}(\mathcal{C})\rightarrow \mathrm{Mod}(\mathcal{C}/\mathcal{I})$ the $i$-th right derived functor of $\mathcal{C}(\frac{\mathcal{C}}{\mathcal{I}},-)$ and by $\mathbb{TOR}_{i}^{\mathcal{C}^{op}}(\mathcal{C}/\mathcal{I},-):\mathrm{Mod}(\mathcal{C})\rightarrow \mathrm{Mod}(\mathcal{C}/\mathcal{I})$ the $i$-th left derived functor of $ \frac{\mathcal{C}}{\mathcal{I}}\otimes_{\mathcal{C}^{op}}$.
\end{definition}

We have the following description of the above functors
\begin{remark}\label{descEXT}
Consider the functors $\mathbb{EXT}^{i}_{\mathcal{C}}(\mathcal{C}/\mathcal{I},-):\mathrm{Mod}(\mathcal{C})\longrightarrow \mathrm{Mod}(\mathcal{C}/\mathcal{I})$ and $\mathbb{TOR}_{i}^{\mathcal{C}^{op}}(\mathcal{C}/\mathcal{I},-):\mathrm{Mod}(\mathcal{C})\longrightarrow \mathrm{Mod}(\mathcal{C}/\mathcal{I})$. The following  holds true.
\begin{enumerate}
\item [(a)] For $M\in\mathrm{Mod}(\mathcal{C})$ we get that $\mathbb{EXT}^{i}_{\mathcal{C}}(\mathcal{C}/\mathcal{I},M)(C)=\mathrm{Ext}^{i}_{\mathrm{Mod}(\mathcal{C})}\left(\frac{\mathrm{Hom}_{\mathcal{C}}(C,-)}{\mathcal{I}(C,-)},M\right)$ for every $C\in \mathcal{C}/\mathcal{I}$.

\item [(b)] For $M\in\mathrm{Mod}(\mathcal{C})$ we have that $\mathbb{TOR}_{i}^{\mathcal{C}^{op}}(\mathcal{C}/\mathcal{I},M)(C)=\mathrm{Tor}_{i}^{\mathcal{C}^{op}}\left(\frac{\mathrm{Hom}_{\mathcal{C}}(-,C)}{\mathcal{I}(-,C)},M\right)$ for every $C\in \mathcal{C}/\mathcal{I}$. 

\end{enumerate}
\end{remark}
From Section 5 in \cite{RSS}, we obtain the following definition, which is a generalization of a notion given for artin algebras by Auslander-Platzeck-Todorov in  \cite{APG}. This notion also appears in  \cite{KoenigNagase} under the name of $\textbf{stratifying ideal}$.

\begin{definition}$\textnormal{\cite[Definition 5.1]{RSS}}$ \label{kidemcat}
Let $\mathcal{C}$ be a $K$-category  and $\mathcal{I}$ an ideal in $\mathcal{C}$.
We say that $\mathcal{I}$ is $\textbf{strongly idempotent}$ if 
$$\varphi^{i}_{F,(\pi_{1})_{\ast}(F')}:\mathrm{Ext}^{i}_{\mathrm{Mod}(\mathcal{C}/\mathcal{I})}(F,F')\longrightarrow \mathrm{Ext}^{i}_{\mathrm{Mod}(\mathcal{C})}((\pi_{1})_{\ast}(F),(\pi_{1})_{\ast}(F'))$$ is an isomorphism for all $F,F'\in \mathrm{Mod}(\mathcal{C}/\mathcal{I})$ and for all $0\leq i < \infty$.
\end{definition}

Now let us consider $\pi_{1}:\mathcal{C}\longrightarrow \mathcal{C}/\mathcal{I}$ and $\pi_{2}:\mathcal{C}^{op}\longrightarrow \mathcal{C}^{op}/\mathcal{I}^{op}$ the canonical projections.
From section 5 in \cite{RSS}, for $F\in \mathrm{Mod}((\mathcal{C}/\mathcal{I})^{op})$  and  $F'\in \mathrm{Mod}(\mathcal{C}/\mathcal{I})$ we have the morphism $\psi_{F,(\pi_{1})_{\ast}(F')}^{i}:\mathrm{Tor}^{\mathcal{C}^{op}}_{i}(F\circ \pi_{2},F'\circ \pi_{1})\longrightarrow \mathrm{Tor}^{(\mathcal{C}/\mathcal{I})^{op}}_{i}(F, F')$. By using that for $N\in \mathrm{Mod}(\mathcal{C}^{op})$ and $M\in \mathrm{Mod}(\mathcal{C})$ there is an isomorphism  $\mathrm{Hom}_{K}\Big(\mathrm{Tor}_{i}^{\mathcal{C}^{op}}(N,M),K\Big)\simeq 
\mathrm{Ext}_{\mathrm{Mod}(\mathcal{C})}^{i}(M,\mathbb{D}_{\mathcal{C}^{op}}(N))$ for all $i\geq 0$,
we obtain the following result that is a kind of generalization of Theorem 4.4 of Geigle and Lenzing in \cite{GeigleLen}.

\begin{proposition}\label{caractidem2}
Let $\mathcal{C}$ be a  $K$-category  and $\mathcal{I}$ an ideal.  The following are equivalent.
\begin{enumerate}
\item [(a)] $\mathcal{I}$ is strongly idempotent 

\item [(b)] $\mathbb{EXT}^{i}_{\mathcal{C}}(\mathcal{C}/\mathcal{I},F'\circ \pi_{1})=0$ for $1\leq i<\infty$ and for $F'\in \mathrm{Mod}(\mathcal{C}/\mathcal{I})$.

\item [(c)] $\mathbb{EXT}^{i}_{\mathcal{C}}(\mathcal{C}/\mathcal{I},J\circ \pi_{1})=0$ for $1\leq i<\infty$ and for each $J\in \mathrm{Mod}(\mathcal{C}/\mathcal{I})$ which is injective.

\item [(d)] $\psi_{F,(\pi_{1})_{\ast}(F')}^{i}:\mathrm{Tor}^{\mathcal{C}^{op}}_{i}(F\circ \pi_{2},F'\circ \pi_{1})\longrightarrow \mathrm{Tor}^{(\mathcal{C}/\mathcal{I})^{op}}_{i}(F, F')$ is an isomorphism for all $0\leq i<\infty$ and $F\in \mathrm{Mod}((\mathcal{C}/\mathcal{I})^{op})$ as well as  $F'\in \mathrm{Mod}(\mathcal{C}/\mathcal{I})$.

\item [(e)] $\mathbb{TOR}_{i}^{\mathcal{C}^{op}}(\mathcal{C}/\mathcal{I},F'\circ\pi_{1})=0$ for $1\leq i<\infty$  and for all $F'\in \mathrm{Mod}(\mathcal{C}/\mathcal{I})$.

\item [(f)] $\mathbb{TOR}_{i}^{\mathcal{C}^{op}}(\mathcal{C}/\mathcal{I},P\circ\pi_{1})=0$ for $1\leq i<\infty$ and for all $P\in \mathrm{Mod}(\mathcal{C}/\mathcal{I})$ which is projective.

\end{enumerate}
\end{proposition}
\begin{proof}
The proof given in \cite[Corollary 5.10]{RSS} can be adapted to this setting.
\end{proof}

The following is a generalization \cite[Definition 4.5]{GeigleLen}.

\begin{definition}
Let $\mathcal{I}$ be an ideal of $\mathcal{C}$. It is said that $\pi_{1}:\mathcal{C}\longrightarrow\mathcal{C}/\mathcal{I}$ is an $\textbf{homological epimorphism}$ if $\mathcal{I}$ is strongly idempotent.
\end{definition}


\begin{proposition}\label{resolproyten}
Let $\mathcal{C}$ and $\mathcal{D}$ be two  $K$-categories, and let $\mathcal{A}$ be an $AB4$ and $AB4^{\ast}$ $K$-category. Let $X^{\bullet}$ be a projective resolution for $F\in \mathrm{Fun}_{K}(\mathcal{C},\mathcal{A})$ and let $Y^{\bullet}$ be a projective resolution for $G\in \mathrm{Mod}(\mathcal{C}^{op}\otimes_{K}\mathcal{D})$. If $\mathrm{Tor}_{n}^{\mathcal{C}}(F,G(-,D))=0$ for all $D\in \mathcal{D}$ and for all $n>0$, then $X^{\bullet}\boxtimes_{\mathcal{C}}Y^{\bullet}$ is a projective resolution of $F\boxtimes_{\mathcal{C}}G\in \mathrm{Fun}_{K}(\mathcal{D},\mathcal{A})$.
\end{proposition}
\begin{proof}
See \cite[Proposition 11.8]{Mitchelring} in p. 55.
\end{proof}
We recall the following result given in formula $(5)$ in page 28 in \cite{Mitchelring}.
Let $F\in \mathrm{Fun}_{K}\Big(\mathcal{E}^{op}\otimes_{K}\mathcal{C},\mathcal{A}\Big)$ be $G\in \mathrm{Mod}(\mathcal{C}^{op}\otimes_{K}\mathcal{D})$ and $H\in \mathrm{Mod}(\mathcal{D}^{op}\otimes_{K}\mathcal{E})$. We then have a natural isomorphism in $\mathcal{A}$:
\begin{equation}
\big(F\boxtimes_{\mathcal{C}}G\big)\otimes_{\mathcal{D}\otimes_{K} \mathcal{E}^{op}} H\simeq F\otimes_{\mathcal{E}^{op}\otimes_{K}\mathcal{C}}\big( G\boxtimes_{\mathcal{D}} H\big).\label{asociatensor}
\end{equation}

The following proposition is a generalization of Theorem 2.8 in page 167 in Cartan and Eilenberg's book \cite{Cartan}.

\begin{proposition}\label{isocartan}
Let $\mathcal{C},\mathcal{D}$ and $\mathcal{E}$ be three $K$-categories. Consider $F\in \mathrm{Fun}_{K}(\mathcal{E}^{op}\otimes_{K} \mathcal{C},\mathcal{A})$,  $G\in \mathrm{Fun}_{K}(\mathcal{C}^{op}\otimes_{K} \mathcal{D},\mathrm{Mod}(K))$ and  $H\in \mathrm{Fun}_{K}(\mathcal{D}^{op}\otimes_{K} \mathcal{E},\mathrm{Mod}(K))$. Suppose that
$\mathrm{Tor}_{n}^{\mathcal{C}}(F,G(-,D))=0$  for all  $D\in \mathcal{D}$ and $\forall n>0$  and  that $\mathrm{Tor}_{n}^{\mathcal{D}}(G,H(-,E))=0$ for all $E\in \mathcal{E}$ and $\forall n>0$. Hence, there exists an isomorphism for all $i\geq 0$:
$$\mathrm{Tor}_{i}^{\mathcal{D}\otimes_{K} \mathcal{E}^{op}}(F\boxtimes_{\mathcal{C}}G,H)\simeq \mathrm{Tor}_{i}^{\mathcal{E}^{op}\otimes_{K} \mathcal{C}}(F,G\boxtimes_{\mathcal{D}}H).$$
\end{proposition}

\begin{proof}
Let  $X^{\bullet}$ be a projective resolution of $F$, with $F$ seen as a functor $F:\mathcal{C}\longrightarrow \mathrm{Fun}_{K}(\mathcal{E}^{op},\mathcal{A})$, and $Y^{\bullet}$ a projective resolution of $G$ in $\mathrm{Mod}(\mathcal{C}^{op}\otimes_{K} \mathcal{D})=\mathrm{Fun}_{K}(\mathcal{C}^{op}\otimes \mathcal{D},\bf{Ab})$. By Proposition \ref{resolproyten} we have that $X^{\bullet}\boxtimes_{\mathcal{C}}Y^{\bullet}$ is a projective resolution of $F\boxtimes_{\mathcal{C}}G$ in $\mathrm{Fun}_{K}(\mathcal{D},\mathrm{Fun}_{K}(\mathcal{E}^{op},\mathcal{A})=\mathrm{Fun}_{K}(\mathcal{E}^{op}\otimes_{K} \mathcal{D},\mathcal{A})$.\\
On the other hand, by considering $G$ as a functor in  $\mathrm{Fun}_{K}(\mathcal{D},\mathrm{Fun}_{K}(\mathcal{C}^{op},\bf{Ab}))$, we have that $Y^{\bullet}$ is a projective resolution of $G:\mathcal{D}\longrightarrow \mathrm{Fun}_{K}(\mathcal{C}^{op},\bf{Ab})$ and let $Z^{\bullet}$ be a projective resolution of $H$ in $ \mathrm{Mod}(\mathcal{D}^{op}\otimes_{K} \mathcal{E})$. Hence, by Proposition \ref{resolproyten} we have that  $Y^{\bullet}\boxtimes_{\mathcal{D}}  Z^{\bullet}$ is a projective resolution of $G\boxtimes_{\mathcal{D}}H$ in $\mathrm{Mod}(\mathcal{C}^{op}\otimes _{K}\mathcal{E})$. By the associativity given above in Equation \ref{asociatensor}, we obtain an isomorphism of complexes in $\mathcal{A}$:
$$(\ast):\big(X^{\bullet}\boxtimes_{\mathcal{C}}Y^{\bullet}\big)\otimes_{\mathcal{D}\otimes \mathcal{E}^{op}} Z^{\bullet}=X^{\bullet}\otimes_{\mathcal{E}^{op}\otimes\mathcal{C}}\big( Y^{\bullet}\boxtimes_{\mathcal{D}} Z^{\bullet}\big).$$
Now, since $X^{\bullet}\boxtimes_{\mathcal{C}} Y^{\bullet}$ is an acyclic complex over $F\boxtimes_{\mathcal{C}}G$ and $Z^{\bullet}$ is a projective resolution of $H$ in $\mathrm{Mod}(\mathcal{D}^{op}\otimes_{K} \mathcal{E})$ and from the discusion on page 32 in \cite{Mitchelring}, we have that
$$\mathrm{Tor}_{i}^{\mathcal{D}\otimes \mathcal{E}^{op}}(F\boxtimes_{\mathcal{C}}G,H)=\mathrm{H}_{i}\Big((X^{\bullet}\boxtimes_{\mathcal{C}} Y^{\bullet})\otimes_{\mathcal{D}\otimes \mathcal{E}^{op}} Z^{\bullet}\Big).$$
Similarly,  we have that $\mathrm{Tor}_{i}^{\mathcal{E}^{op}\otimes\mathcal{C}}\Big(F,G\otimes_{\mathcal{D}}H\Big)=\mathrm{H}_{i}\Big(X^{\bullet}\otimes_{\mathcal{E}^{op}\otimes\mathcal{C}}\big( Y^{\bullet}\boxtimes_{\mathcal{D}} Z^{\bullet}\big)\Big).$ Therefore, by the isomorphism of complexes $(\ast)$ we conclude the proof.
\end{proof}

\begin{proposition}\label{pipiophomo}
Let $\mathcal{C}$ be a $K$-category and $\mathcal{I}$ an ideal of $\mathcal{C}$. Thus, $\pi_{1}:\mathcal{C}\longrightarrow \mathcal{C}/\mathcal{I}$ is an homological epimorphism if and only if $\pi_{2}:\mathcal{C}^{op}\longrightarrow \mathcal{C}^{op}/\mathcal{I}^{op}$ is an homological epimorphism.
\end{proposition}
\begin{proof}
$(\Rightarrow)$. Suppose that $\pi_{1}:\mathcal{C}\longrightarrow \mathcal{C}/\mathcal{I}$ is an homological epimorphism.\\
Let us consider $\pi_{2}:\mathcal{C}^{op}\longrightarrow \mathcal{C}^{op}/\mathcal{I}^{op}$. By Proposition \ref{caractidem2}, we must see that  $\mathbb{EXT}^{i}_{\mathcal{C}^{op}}(\mathcal{C}^{op}/\mathcal{I}^{op},F'\circ \pi_{2})=0$ for $1\leq i< \infty$ and for $F'\in \mathrm{Mod}(\mathcal{C}^{op}/\mathcal{I}^{op})$. That is, for $C\in \mathcal{C}$ we have to see that $\mathrm{Ext}^{i}_{\mathrm{Mod}(\mathcal{C}^{op})}\Big(\frac{\mathrm{Hom}_{\mathcal{C}}(-,C)}{\mathcal{I}(-,C)},F'\circ \pi_{2}\Big)=0.$
Consider the canonical functor $(\pi_{2})_{\ast}:\mathrm{Mod}(\mathcal{C}^{op}/\mathcal{I}^{op})\longrightarrow \mathrm{Mod}(\mathcal{C}^{op}).$
Since $\pi_{2}$ is an epimorphism in the category of functors, we have that
$$\mathrm{Hom}_{\mathrm{Mod}(\mathcal{C}^{op})}(X\circ \pi_{2}, Y\circ \pi_{2})\simeq \mathrm{Hom}_{\mathrm{Mod}(\mathcal{C}^{op}/\mathcal{I}^{op})}(X,Y)$$
for all $X,Y\in \mathrm{Mod}(\mathcal{C}^{op}/\mathcal{I}^{op})$. We now recall that $\mathrm{Hom}_{\mathcal{C}^{op}/\mathcal{I}^{op}}(-,C)\circ \pi_{2}=\frac{\mathrm{Hom}_{\mathcal{C}}(-,C)}{\mathcal{I}(-,C)}\in \mathrm{Mod}(\mathcal{C}^{op})$ (see  \cite[Lemma 3.7c]{RSS}).\\
Let $\xymatrix{0\ar[r] & L\ar[r]^{\alpha} & M\ar[r]^{\beta} & N\ar[r] & 0}$ be an exact sequence in $\mathrm{Mod}(\mathcal{C}^{op}/\mathcal{I}^{op})$. 
Since $\mathrm{Hom}_{\mathcal{C}^{op}/\mathcal{I}^{op}}(-,C)$ is projective in $\mathrm{Mod}(\mathcal{C}^{op}/\mathcal{I}^{op})$,  we conclude that  we have the following exact sequence in $\mathrm{Mod}(K)$:
$$0\rightarrow \Big(\frac{\mathrm{Hom}_{\mathcal{C}}(-,C)}{\mathcal{I}(-,C)}, L\circ \pi_{2}\Big)\rightarrow \Big(\frac{\mathrm{Hom}_{\mathcal{C}}(-,C)}{\mathcal{I}(-,C)}, M\circ \pi_{2}\Big)\rightarrow \Big(\frac{\mathrm{Hom}_{\mathcal{C}}(-,C)}{\mathcal{I}(-,C)}, N\circ \pi_{2}\Big)\rightarrow 0.$$
Now let us see by induction on $i$ that 
$\mathrm{Ext}^{i}_{\mathrm{Mod}(\mathcal{C}^{op})}\Big(\frac{\mathrm{Hom}_{\mathcal{C}}(-,C)}{\mathcal{I}(-,C)},F'\circ \pi_{2}\Big)=0$ for $F'\in \mathrm{Mod}(\mathcal{C}^{op}/\mathcal{I}^{op})$.\\
We have the following exact sequence in $\mathrm{Mod}(\mathcal{C}^{op})$: 
$$\xymatrix{0\ar[r] & F'\circ \pi_{2}\ar[r]^(.4){\varphi} & \mathbb{D}_{\mathcal{C}}\mathbb{D}_{\mathcal{C}^{op}}(F'\circ \pi_{2})\ar[r]^(.6){\psi} & L\ar[r] & 0.}$$
It is easy to see that there exists $L\in \mathrm{Mod}(\mathcal{C}^{op}/\mathcal{I}^{op})$ such that $L\simeq T\circ \pi_{2}$.
By applying $\mathrm{Hom}_{\mathrm{Mod}(\mathcal{C}^{op})}\Big(\frac{\mathrm{Hom}_{\mathcal{C}}(-,C)}{\mathcal{I}(-,C)},-\Big)$ to the the last exact sequence and by the long exact sequence in homology we obtain a monomorphism

$$\xymatrix{\mathrm{Ext}_{\mathrm{Mod}(\mathcal{C}^{op})}^{1}\Big(\frac{\mathrm{Hom}_{\mathcal{C}}(-,C)}{\mathcal{I}(-,C)}, F'\circ \pi_{2}\Big)\ar[r] &  \mathrm{Ext}_{\mathrm{Mod}(\mathcal{C}^{op})}^{1}\Big(\frac{\mathrm{Hom}_{\mathcal{C}}(-,C)}{\mathcal{I}(-,C)}, \mathbb{D}_{\mathcal{C}}\mathbb{D}_{\mathcal{C}^{op}}(F'\circ \pi_{2})\Big). }$$

We assert that $\mathrm{Ext}_{\mathrm{Mod}(\mathcal{C}^{op})}^{1}\Big(\frac{\mathrm{Hom}_{\mathcal{C}}(-,C)}{\mathcal{I}(-,C)},  \mathbb{D}_{\mathcal{C}}\mathbb{D}_{\mathcal{C}^{op}}(F'\circ \pi_{2})\Big)=0.$ Indeed, we  can consider the following isomorphisms $\mathrm{Ext}^{i}_{\mathrm{Mod}(\mathcal{C}^{op})}\Big(\frac{\mathrm{Hom}_{\mathcal{C}}(-,C)}{\mathcal{I}(-,C)},\mathbb{D}_{\mathcal{C}}\mathbb{D}_{\mathcal{C}^{op}}(F'\circ \pi_{2})\Big)\simeq$ $
\mathrm{Hom}_{K}\Big(\mathrm{Tor}^{\mathcal{C}^{op}}_{i}\Big(\frac{\mathrm{Hom}_{\mathcal{C}}(-,C)}{\mathcal{I}(-,C)},\mathbb{D}_{(\mathcal{C}/\mathcal{I})^{op}}(F')\circ \pi_{1}\Big),K\Big)$ for all $i\geq 1$.\\
Since $\pi_{1}$ is a homological epimorphism, we get that $\mathrm{Tor}^{C^{op}}_{i}\Big(\frac{\mathrm{Hom}_{\mathcal{C}}(-,C)}{\mathcal{I}(-,C)},\mathbb{D}_{(\mathcal{C}/\mathcal{I})^{op}}(F')\circ \pi_{1}\Big)=0$ (see Proposition \ref{caractidem2}); and hence  $\mathrm{Ext}_{\mathrm{Mod}(\mathcal{C}^{op})}^{1}\Big(\frac{\mathrm{Hom}_{\mathcal{C}}(-,C)}{\mathcal{I}(-,C)},  \mathbb{D}_{\mathcal{C}}\mathbb{D}_{\mathcal{C}^{op}}(F'\circ \pi_{2})\Big)=0.$ Thus, we conclude that $\mathrm{Ext}_{\mathrm{Mod}(\mathcal{C}^{op})}^{1}\Big(\frac{\mathrm{Hom}_{\mathcal{C}}(-,C)}{\mathcal{I}(-,C)}, F'\circ \pi_{2}\Big)=0.$
Similarly, we can see that $\mathrm{Ext}_{\mathrm{Mod}(\mathcal{C}^{op})}^{1}\Big(\frac{\mathrm{Hom}_{\mathcal{C}}(-,C)}{\mathcal{I}(-,C)},L\Big)=0.$
We can proceed as above for each $i$ and prove that $\mathrm{Ext}_{\mathrm{Mod}(\mathcal{C}^{op})}^{i}\Big(\frac{\mathrm{Hom}_{\mathcal{C}}(-,C)}{\mathcal{I}(-,C)}, F'\circ \pi_{2}\Big)=0.$ Hence, by  Proposition \ref{caractidem2} we conclude that $\pi_{2}$ is a homological epimorphism. The other implication is similar.
\end{proof}

We omit the proof of the following Lemma.

\begin{lemma}\label{tensordomhom}
Consider $\pi_{1}:\mathcal{C}\longrightarrow \mathcal{C}/\mathcal{I}$ and $\pi_{2}\otimes\pi_{1}:\mathcal{C}^{op}\otimes_{K}\mathcal{C} \longrightarrow (\mathcal{C}/\mathcal{I})^{op}\otimes_{K}(\mathcal{C}/\mathcal{I})=\mathcal{D}$. We then have the following isomorphisms in $\mathrm{Mod}(\mathcal{C}^{op}\otimes_{K}\mathcal{C})$:
\begin{enumerate}
\item [(a)] $\overline{\mathrm{Hom}_{\mathcal{C}/\mathcal{I}}(-,U')\circ \pi_{2}}\boxtimes_{\mathcal{C}_{K}} \underline{\mathrm{Hom}_{\mathcal{C}/\mathcal{I}}(U,-)\circ \pi_{1}}\simeq \mathrm{Hom}_{\mathcal{D}}\Big(-,(U,U')\Big)\circ (\pi_{2}\otimes \pi_{1})$ for all $(U,U')\in \mathcal{C}^{op}\otimes_{K}\mathcal{C}$ and,

\item [(b)] $\overline{\mathrm{Hom}_{\mathcal{C}/\mathcal{I}}(-,Z)\circ \pi_{2}}\boxtimes_{\mathcal{C}_{K}} \underline{\mathrm{Hom}_{\mathcal{C}/\mathcal{I}}(Z',-)\circ \pi_{1}}\simeq \mathrm{Hom}_{\mathcal{D}}\Big((Z,Z'),-\Big)\circ (\pi_{2}\otimes \pi_{1})$ for all $(Z,Z')\in \mathcal{C}^{op}\otimes_{K}\mathcal{C}$.
\end{enumerate}

\end{lemma}
\begin{proof}
It is Straightforward.
\end{proof}

\begin{proposition}\label{piehomoloepi}
Let $\mathcal{C}$ be a $K$-category and  $\pi_{1}:\mathcal{C} \longrightarrow \mathcal{C} /\mathcal{I}$ be a homological epimorphism. Thus,
$\pi^{e}:=\pi_{2}\otimes\pi_{1}:\mathcal{C}^{op}\otimes_{K}\mathcal{C}\longrightarrow (\mathcal{C}/\mathcal{I})^{op}\otimes_{K}\mathcal{C}/\mathcal{I}$ is a homological epimorphism.
\end{proposition}
\begin{proof}
Let $\mathcal{J}:=\mathrm{Ker}(\pi^{e})$ be, we can identify $\pi^{e}$ with the canonical epimorphism $\Pi: \mathcal{C}^{e}\longrightarrow \mathcal{C}^{e}/\mathcal{J}.$
Then, it is then sufficient to see that $\Pi: \mathcal{C}^{e}\longrightarrow \mathcal{C}^{e}/\mathcal{J}$ is a homological epimorphism.
For an object $X:=(U,U')\in (\mathcal{C}^{op}\otimes_{K}\mathcal{C})/\mathcal{J}=\mathcal{C}^{e}/\mathcal{J}\simeq (\mathcal{C}/\mathcal{I})^{op}\otimes_{K}(\mathcal{C}/\mathcal{I})$, we have an isomorphism  by Lemma \ref{tensordomhom} (a):
\begin{align*}
\overline{\mathrm{Hom}_{\mathcal{C}/\mathcal{I}}(-,U')\circ \pi_{2}}\boxtimes_{\mathcal{C}_{K}} \underline{\mathrm{Hom}_{\mathcal{C}/\mathcal{I}}(U,-)\circ \pi_{1}} & \simeq \mathrm{Hom}_{\mathcal{C}^{e}/\mathcal{J}}\Big(-,(U,U')\Big)\circ (\pi^{e})\\& \simeq \mathrm{Hom}_{\mathcal{C}^{e}}(-,X)/\mathcal{J}(-,X).
\end{align*}
By Proposition \ref{caractidem2}, we must to show that $\mathbb{TOR}_{i}^{(\mathcal{C}^{e})^{op}}\Big(\mathcal{C}^{e}/\mathcal{J},P\circ \Pi\Big)=0$ for all $P\in \mathrm{Mod}(\mathcal{C}^{e}/\mathcal{J})$ that is projective. It is enough to take $P:=\mathrm{Hom}_{\mathcal{C}^{e}/\mathcal{J}}\Big(Y,-\Big)$ for $Y\in \mathcal{C}^{e}/\mathcal{J}$. We then have that
\begin{align*}
& \mathbb{TOR}_{i}^{(\mathcal{C}^{e})^{op}}\Big(\mathcal{C}^{e}/\mathcal{J},\mathrm{Hom}_{\mathcal{C}^{e}}(Y,-)/\mathcal{J}(Y,-)\Big)(X)=\\
& =\mathrm{Tor}_{i}^{\mathcal{C}\otimes_{K}\mathcal{C}^{op}}\Big(\mathrm{Hom}_{\mathcal{C}^{e}}(-,X)/\mathcal{J}(-,X),\mathrm{Hom}_{\mathcal{C}^{e}}(Y,-)/\mathcal{J}(Y,-)\Big)\\
& =\mathrm{Tor}_{i}^{\mathcal{C}\otimes_{K}\mathcal{C}^{op}}\!\Big(\overline{\mathrm{Hom}_{\mathcal{C}/\mathcal{I}}(-,U')\circ \!\pi_{2}}\boxtimes_{\mathcal{C}_{K}}\underline{\mathrm{Hom}_{\mathcal{C}/\mathcal{I}}(U,-)\circ\! \pi_{1}}
, \mathrm{Hom}_{\mathcal{C}^{e}}(Y,-)/\mathcal{J}(Y,-)\Big)
\end{align*}
Let us check that the hypothesis of Proposition \ref{isocartan} holds.\\
Firstly, $\mathrm{Tor}_{n}^{\mathcal{C}_{K}}\Big(\overline{\mathrm{Hom}_{\mathcal{C}/\mathcal{I}}(-,U')\circ \pi_{2}},(\underline{\mathrm{Hom}_{\mathcal{C}/\mathcal{I}}(U,-)\circ \pi_{1}})(-,D)\Big)=0$ for all $D\in \mathcal{C}$ since $K$ is a field.\\
We will now show that $\mathrm{Tor}_{i}^{\mathcal{C}}\!\Big(\!(\underline{\mathrm{Hom}_{\mathcal{C}/\mathcal{I}}(U,-)\!\circ \!\pi_{1}}),\mathrm{Hom}_{\mathcal{C}^{e}}(Y,-)/\mathcal{J}(Y,-)(-,E')\Big)=0$
for all $E'\in \mathcal{C}$.\\
It can be easily seen that there exists $W\in \mathrm{Mod}(\mathcal{C}^{op}/\mathcal{I}^{op})$ such that $W\circ \pi_{2}\simeq \mathrm{Hom}_{\mathcal{C}^{e}}(Y,-)/\mathcal{J}(Y,-)(-,E').$ Hence, we obtain
\begin{align*}
\mathrm{Tor}_{i}^{\mathcal{C}}\!\Big(\underline{\mathrm{Hom}_{\mathcal{C}/\mathcal{I}}(U,-)\!\circ \!\pi},\mathrm{Hom}_{\mathcal{C}^{e}}(Y,-)/\mathcal{J}(Y,-)(-,E')\!\Big) & \!\!\simeq\!\mathrm{Tor}_{i}^{\mathcal{C}}\Big(\frac{\mathrm{Hom}_{\mathcal{C}}(U,-)}{\mathcal{I}(U,-)},W\!\!\circ\!\pi_{2}\!\Big)\\
& =0,
\end{align*}
where the last equality is because  $\pi_{2}$ is a homological epimorphism. Therefore, by Proposition \ref{isocartan} we have that
\begin{align*}
& \mathrm{Tor}_{i}^{\mathcal{C}\otimes_{K}\mathcal{C}^{op}}\Big(\overline{\mathrm{Hom}_{\mathcal{C}/\mathcal{I}}(-,U')\!\circ \!\pi_{2}}\boxtimes_{\mathcal{C}_{K}}\underline{\mathrm{Hom}_{\mathcal{C}/\mathcal{I}}(U,-)\circ \pi_{1}}
,\mathrm{Hom}_{\mathcal{C}^{e}}(Y,-)/\mathcal{J}(Y,-)\Big)\!=\\
&\! =\!\mathrm{Tor}_{i}^{\mathcal{C}^{op}\otimes \mathcal{C}_{K}}\!\Big(\overline{\mathrm{Hom}_{\mathcal{C}/\mathcal{I}}(-,U')\circ \pi_{2}},\,\,\,\,\underline{\mathrm{Hom}_{\mathcal{C}/\mathcal{I}}(U,-)\circ \pi_{1}}\boxtimes_{\mathcal{C}}
\mathrm{Hom}_{\mathcal{C}^{e}}(Y,-)/\mathcal{J}(Y,-)\!\Big)\\
& \!=\!\mathrm{Tor}_{i}^{\mathcal{C}^{op}}\!\Big(\mathrm{Hom}_{\mathcal{C}/\mathcal{I}}(-,U')\circ \pi_{2},\,\,\,\big(\mathrm{Hom}_{\mathcal{C}/\mathcal{I}}(U,-)\circ \pi_{1}\big)\boxtimes_{\mathcal{C}}
\big(\mathrm{Hom}_{\mathcal{C}^{e}}(Y,-)/\mathcal{J}(Y,-)\big)\!\Big)
\end{align*}
It can be seen that there exists $Q\in \mathrm{Mod}(\mathcal{C}/ \mathcal{I})$ such that $Q\circ \pi_{1}\simeq \big(\mathrm{Hom}_{\mathcal{C}/\mathcal{I}}(U,-)\circ \pi_{1}\big)\boxtimes_{\mathcal{C}}
\big(\mathrm{Hom}_{\mathcal{C}^{e}}(Y,-)/\mathcal{J}(Y,-)\big)$.
Therefore,

\begin{align*}
& \mathrm{Tor}_{i}^{\mathcal{C}^{op}}\Big(\mathrm{Hom}_{\mathcal{C}/\mathcal{I}}(-,U')\circ \pi_{2},\,\,\,\big(\mathrm{Hom}_{\mathcal{C}/\mathcal{I}}(U,-)\circ \pi_{1}\big)\boxtimes_{\mathcal{C}}
\big(\mathrm{Hom}_{\mathcal{C}^{e}}(Y,-)/\mathcal{J}(Y,-)\big)\Big)\\
& \simeq  \mathrm{Tor}_{i}^{\mathcal{C}^{op}}\Big(\mathrm{Hom}_{\mathcal{C}/\mathcal{I}}(-,U')\circ \pi_{2},\,\,\,  Q\circ \pi_{1}\Big) \simeq  \mathrm{Tor}_{i}^{\mathcal{C}^{op}}\Big(\frac{\mathrm{Hom}_{\mathcal{C}}(-,U')}{\mathcal{I}(-,U')},\,\,\,  Q\circ \pi_{1}\Big)=0,
\end{align*}
where the last equality holds because $\pi_{1}$ is a homological epimorphism and because of Proposition \ref{caractidem2}. \\
We have proven that $\mathbb{TOR}_{i}^{(\mathcal{C}^{e})^{op}}\Big(\mathcal{C}^{e}/\mathcal{J},\mathrm{Hom}_{\mathcal{C}^{e}}(Y,-)/\mathcal{J}(Y,-)\Big)(X)=0.$
We then have that $\mathbb{TOR}_{i}^{(\mathcal{C}^{e})^{op}}\Big(\mathcal{C}^{e}/\mathcal{J},\mathrm{Hom}_{\mathcal{C}^{e}}(Y,-)/\mathcal{J}(Y,-)\Big)=0$. By Proposition \ref{caractidem2}, we conclude that 
$\Pi:\mathcal{C}^{e}\longrightarrow \mathcal{C}^{e}/\mathcal{J}$ is a homological epimorphism.
\end{proof}

\section{Homological epimorphisms and Hochschild Mitchell-Cohomology}

Consider $\mathcal{I}$  an ideal of $\mathcal{C}$ and  $\Phi:\mathcal{C}\longrightarrow \mathcal{C}/\mathcal{I}=\mathcal{B}$ the canonical epimorphism. Consider
$H:=\mathcal{B}(-,-)\circ (\Phi^{op}\otimes \Phi)$. Thus, we obtain a morphism in $\mathrm{Mod}(\mathcal{C}^{e})$:
$$\Gamma(\Phi):\mathcal{C}(-,-)\longrightarrow \mathcal{B}(-,-)\circ (\Phi^{op}\otimes \Phi)$$
such that for an object $(C,C')\in \mathcal{C}^{e}$ we have that $[\Gamma(\Phi)]_{(C,C')}:\mathcal{C}(C,C')\longrightarrow \mathcal{B}(\Phi(C),\Phi(C'))$ is defined as $[\Gamma(\Phi)]_{(C,C')}(f):=\Phi(f)$ for all $f\in \mathcal{C}(C,C')$. Thus, we obtain the following exact sequence in $\mathrm{Mod}(\mathcal{C}^{e})$:

\begin{equation}\label{succanonica}
\xymatrix{0\ar[r] & \mathcal{I}\ar[r] & \mathcal{C}\ar[r]^{\Gamma(\Phi)} & H \ar[r]  & 0.}
\end{equation}

Now we have the following result.
\begin{lemma}\label{morfiscerocc}
Let $\mathcal{I}$ be a strongly idempotent ideal of $\mathcal{C}$, $\Phi:\mathcal{C}\longrightarrow \mathcal{C}/\mathcal{I}=\mathcal{B}$ the canonical epimorphism, and consider the following exact sequence in $\mathrm{Mod}(\mathcal{C}^{e})$:
$$\xymatrix{0\ar[r] & \mathcal{I}\ar[r] & \mathcal{C}\ar[r]^{\Gamma(\Phi)} & H \ar[r]  & 0.}$$
Hence, $\mathrm{Ext}_{\mathrm{Mod}(\mathcal{C})}^{i}\Big(\mathcal{I}\Big(C,-\Big),\,\, H\Big(C'',-\Big)\Big)=0$
 for all $C$, $C''\in \mathcal{C}$ and for all $i\geq 0$.
 \end{lemma}
 \begin{proof}
 Firstly, let us see that $\mathrm{Hom}_{\mathrm{Mod}(\mathcal{C})}\Big(\mathcal{I}\Big(C,-\Big),\,\, H\Big(C'',-\Big)\Big)=0$. We note that $H(C'',-)=\frac{\mathcal{C}(C'',-)}{\mathcal{I}(C'',-)}$. For  $C \in \mathcal{C}^{op}$ consider the exact sequence in $\mathrm{Mod}(\mathcal{C})$
$$(\ast):\xymatrix{0\ar[r] & \mathcal{I}(C,-)\ar[r] & \mathcal{C}(C,-)\ar[r]^(.4){\Psi} & H(C,-)=\frac{\mathcal{C}(C,-)}{\mathcal{I}(C,-)}\ar[r] & 0}$$
where $\Psi=[\Gamma(\Phi)]_{(C,-)}$.\\
Since $\mathcal{I}$ is strongly idempotent, we have by  Proposition \ref{caractidem2}(c) that $\mathbb{EXT}^{i}_{\mathcal{C}}\Big(\mathcal{C}/\mathcal{I}, F'\circ \Phi\Big)=0$ for $i\geq 1$ and all $F'\in \mathrm{Mod}(\mathcal{C}/\mathcal{I})$.
 By  Remark \ref{descEXT}(a), for $C\in \mathcal{C}$ we have that
$$\mathbb{EXT}^{i}_{\mathcal{C}}\Big(\mathcal{C}/\mathcal{I}, F'\circ \Phi\Big)(C):=\mathrm{Ext}^{i}_{\mathrm{Mod}(C)}\Big(\frac{\mathcal{C}(C,-)}{\mathcal{I}(C,-)},F'\circ \Phi \Big)=0$$
for all $F'\in \mathrm{Mod}(\mathcal{C}/\mathcal{I})$. 
We then obtain
$$(\star):\mathrm{Ext}^{i}_{\mathrm{Mod}(C)}\Big(\frac{\mathcal{C}(C,-)}{\mathcal{I}(C,-)},\frac{\mathcal{C}(C'',-)}{\mathcal{I}(C'',-)}\Big)=0\quad \forall i\geq 1.$$
Therefore, by applying $\mathrm{Hom}_{\mathrm{Mod}(\mathcal{C})}\Big(-,\frac{\mathcal{C}(C'',-)}{\mathcal{I}(C'',-)}\Big)$ to the exact sequence above $(\ast)$, we have the following exact sequence in $\textbf{Ab}$:
$$(\ast\ast)\!\!:\!0\rightarrow \Big(H(C),\frac{\mathcal{C}(C'',-)}{\mathcal{I}(C'',-)}\Big)\stackrel{\alpha}{\rightarrow}\Big(\mathcal{C}(C,-),\frac{\mathcal{C}(C'',-)}{\mathcal{I}(C'',-)}\Big)\rightarrow  \Big(\mathcal{I}(C,-),\frac{\mathcal{C}(C'',-)}{\mathcal{I}(C'',-)}\Big)\rightarrow 0.$$
We assert that $\alpha:=-\circ \Psi$ is surjective. Indeed, consider $\eta\in  \Big(\mathcal{C}(C,-),\frac{\mathcal{C}(C'',-)}{\mathcal{I}(C'',-)}\Big)$. We have Yoneda's isomorphism $Y: \Big(\mathcal{C}(C,-),\frac{\mathcal{C}(C'',-)}{\mathcal{I}(C'',-)}\Big)\longrightarrow \frac{\mathcal{C}(C'',C)}{\mathcal{I}(C'',C)},$ so, for $\eta\in  \Big(\mathcal{C}(C,-),\frac{\mathcal{C}(C'',-)}{\mathcal{I}(C'',-)}\Big)$, we have that $Y(\eta)=\eta_{C}(1_{C})\in \frac{\mathcal{C}(C'',C)}{\mathcal{I}(C'',C)}$.\\
Since $\Phi$ is an epimorphism, we have that the functor $\Phi_{\ast}:\mathrm{Mod}(\mathcal{C}/\mathcal{I})\longrightarrow \mathrm{Mod}(\mathcal{C})$ is full and faithful (see for example \cite{Heninepi}). Then we have the following isomorphisms:

\begin{align*}
& \mathrm{Hom}_{\mathrm{Mod}(\mathcal{C}/\mathcal{I})}\Big(\mathrm{Hom}_{\mathcal{C}/\mathcal{I}}(C,-),\mathrm{Hom}_{\mathcal{C}/\mathcal{I}}(C'',-)\Big)\simeq \\
 & \simeq \mathrm{Hom}_{\mathrm{Mod}(\mathcal{C})}\Big(\mathrm{Hom}_{\mathcal{C}/\mathcal{I}}(C,-)\circ \Phi,\mathrm{Hom}_{\mathcal{C}/\mathcal{I}}(C'',-)\circ \Phi\Big)\\
& =\mathrm{Hom}_{\mathrm{Mod}(\mathcal{C})}\Big(\frac{\mathcal{C}(C,-)}{\mathcal{I}(C,-)},\frac{\mathcal{C}(C'',-)}{\mathcal{I}(C'',-)}\Big).
\end{align*}

By Yoneda's Lemma we have an isomorphism
$$Y'\!\!:\!\mathrm{Hom}_{\mathrm{Mod}(\mathcal{C}/\mathcal{I})}\Big(\!\!\mathrm{Hom}_{\mathcal{C}/\mathcal{I}}(C,-),\mathrm{Hom}_{\mathcal{C}/\mathcal{I}}(C'',-)\!\Big)\rightarrow \mathrm{Hom}_{\mathcal{C}/\mathcal{I}}(C'',C)\!=\!\frac{\mathcal{C}(C'',C)}{\mathcal{I}(C'',C)}.$$

$Y(\eta)=\eta_{C}(1_{C})\in \frac{\mathcal{C}(C'',C)}{\mathcal{I}(C'',C)}$ then determines  a natural transformation in $\mathrm{Mod}(\mathcal{C}/\mathcal{I})$
$$\eta':\mathrm{Hom}_{\mathcal{C}/\mathcal{I}}(C,-),\longrightarrow \mathrm{Hom}_{\mathcal{C}/\mathcal{I}}(C'',-)$$
such that $Y'(\eta')=\eta_{C}(1_{C})$. That is, $\eta'$ satisfies that $\eta'_{C}(1_{C}+\mathcal{I}(C,C))=\eta_{C}(1_{C})$.
Now consider the following natural transformation 
$$\delta:=\eta'\circ \Phi:\frac{\mathcal{C}(C,-)}{\mathcal{I}(C,-)}=\mathrm{Hom}_{\mathcal{C}/\mathcal{I}}(C,-)\circ \Phi\longrightarrow \frac{\mathcal{C}(C'',-)}{\mathcal{I}(C'',-)}=\mathrm{Hom}_{\mathcal{C}/\mathcal{I}}(C,-)\circ \Phi$$ in $\mathrm{Mod}(\mathcal{C})$. We then have the following natural transformation in $\mathrm{Mod}(\mathcal{C})$:
$$\xymatrix{\mathcal{C}(C,-)\ar[r]^{\Psi} & \frac{\mathcal{C}(C,-)}{\mathcal{I}(C,-)}\ar[r]^{\delta} & \frac{\mathcal{C}(C'',-)}{\mathcal{I}(C'',-)}.}$$


It is easy to see that  $\delta \circ \Psi=\eta$. This proves that $\alpha=-\circ \Psi$ is surjective, and hence from the exact sequence $(\ast\ast)$ we get that  $\mathrm{Hom}_{\mathrm{Mod}(\mathcal{C})}\Big(\mathcal{I}\Big(C,-\Big),\,\, H\Big(C'',-\Big)\Big)=0$.\\
Now, by applying $\mathrm{Hom}_{\mathrm{Mod}(\mathcal{C})}(-,H(C''))$ to the exact sequence $(\ast)$, we obtain the long exact sequence of homology. Furthermore, by using that $\mathcal{C}(C,-)$ is projective in $\mathrm{Mod}(\mathcal{C})$, we conclude for $i\geq 1$ the following isomorphism:
$$\mathrm{Ext}_{\mathrm{Mod}(\mathcal{C})}^{i}\Big(\mathcal{I}(C,-),\frac{\mathcal{C}(C'',-)}{\mathcal{I}(C'',-)}\Big)\simeq \mathrm{Ext}_{\mathrm{Mod}(\mathcal{C})}^{i+1}\Big(H(C),\frac{\mathcal{C}(C'',-)}{\mathcal{I}(C'',-)}\Big)=0,$$
where the last equality is by the equality $(\star)$ above since $\mathcal{I}$ is strongly idempotent.



 \end{proof}

\begin{corollary}\label{HomIH=0}
Let $\mathcal{I}$ be a strongly idempotent ideal of $\mathcal{C}$. Then
\begin{enumerate}
\item [(a)]
We have that $\mathrm{Hom}_{\mathrm{Mod}(\mathcal{C}^{e})}(\mathcal{I},H)=0$.

\item [(b)] Consider the functor $\mathrm{Hom}_{\mathrm{Mod}(\mathcal{C})}(\mathcal{I},H):\mathcal{C}^{op}\otimes_{K}\mathcal{C} \longrightarrow  \bf{Ab}$
defined as
$$(\mathrm{Hom}_{\mathrm{Mod}(\mathcal{C})}(\mathcal{I},H))\Big(C,C''\Big):=\mathrm{Hom}_{\mathrm{Mod}(\mathcal{C})}\Big(\mathcal{I}\Big(C,-\Big),H\Big(C'',-\Big)\Big),$$
for $(C,C'')\in \mathcal{C}^{op}\otimes_{K}\mathcal{C}$.
Thus, $\mathrm{Hom}_{\mathrm{Mod}(\mathcal{C})}(\mathcal{I},H):\mathcal{C}^{op}\otimes_{K}\mathcal{C}\longrightarrow  \bf{Ab}$ is the zero functor.
\end{enumerate}
\end{corollary}
\begin{proof}
$(a)$. Since $\mathrm{Mod}(\mathcal{C}^{e})\simeq \mathrm{Fun}_{K}\Big(\mathcal{C}^{op},\mathrm{Mod}(\mathcal{C})\Big)$, it is sufficient to show that  for each $C \in \mathcal{C}^{op}$ we have $\mathrm{Hom}_{\mathrm{Mod}(\mathcal{C})}\Big(\mathcal{I}\Big(C,-\Big),\,\, H\Big(C,-\Big)\Big)=0.$ This follows from Lemma \ref{morfiscerocc}, however.\\
$(b)$.  This follows from Lemma \ref{morfiscerocc}.
\end{proof}

\begin{proposition}\label{propopartida1new}
Let $\mathcal{I}$ be a strongly idempotent ideal of $\mathcal{C}$.
Thus, we have that
$\mathrm{Hom}_{\mathrm{Mod}(\mathcal{C}^{e})}(\mathcal{C},H)\simeq \mathrm{Hom}_{\mathrm{Mod}(\mathcal{C}
^{e})}(H,H)\simeq \mathrm{Hom}_{\mathrm{Mod}(\mathcal{B}^{e})}(\mathcal{B},\mathcal{B})=H^{0}(\mathcal{B})$.
\end{proposition}
\begin{proof}
Consider the canonical epimorphism $\Phi:\mathcal{C}\longrightarrow \mathcal{C}/\mathcal{I}$.
Since $\Phi^{op}\otimes \Phi$ is an epimorphism, we have that $(\Phi^{op}\otimes\Phi)_{\ast}:\mathrm{Mod}(\mathcal{B}^{e})\longrightarrow \mathrm{Mod}(\mathcal{C}^{e})$ is a functor that is full and faithful (see for example \cite{Heninepi}). There is then an isomorphism
\begin{align*}
\mathrm{Hom}_{\mathrm{Mod}(\mathcal{B}^{e})}(\mathcal{B},\mathcal{B}) & \simeq \mathrm{Hom}_{\mathrm{Mod}(\mathcal{C}^{e})}\Big(\mathcal{B}(-,-)\circ (\Phi^{op}\otimes \Phi),\,\,\mathcal{B}(-,-)\circ (\Phi^{op}\otimes \Phi)\Big)\\& =\mathrm{Hom}_{\mathrm{Mod}(\mathcal{C}^{e})}(H,H).
\end{align*}
By Corollary \ref{HomIH=0}(a), we have that $\mathrm{Hom}_{\mathrm{Mod}(\mathcal{C}^{e})}(\mathcal{I},H)=0.$  Then, by applying the functor $\mathrm{Hom}_{\mathrm{Mod}(\mathcal{C}^{e})}(-,H)$ to the exact sequence \ref{succanonica}, we get the exact sequence
$$\xymatrix{0\ar[r] & \mathrm{Hom}_{\mathrm{Mod}(\mathcal{C}^{e})}(H,H)\ar[r] & \mathrm{Hom}_{\mathrm{Mod}(\mathcal{C}^{e})}(\mathcal{C},H)\ar[r] & \mathrm{Hom}_{\mathrm{Mod}(\mathcal{C}^{e})}(\mathcal{I},H)=0.}$$
Hence, we obtain an isomorphism $\mathrm{Hom}_{\mathrm{Mod}(\mathcal{C}^{e})}(\mathcal{C},H)\simeq \mathrm{Hom}_{\mathrm{Mod}(\mathcal{C}^{e})}(H,H)$.
\end{proof}

The following Proposition give us a criterion to decide when an ideal is strongly idempotent.

\begin{proposition}\label{proyepihomo}
Let $\mathcal{I}$ be an idempotent ideal of $\mathcal{C}$ such that $\mathcal{I}(C,-)$ is projective in $\mathrm{Mod}(\mathcal{C})$ for all $C\in \mathcal{C}$.  Then $\mathcal{I}$ is strongly idempotent.
\end{proposition}
\begin{proof}
Since $\mathcal{I}$ is idempotent (=$1$-idempotent), by \cite[Proposition 5.3]{RSS} we have that
$$0=\mathrm{Ext}^{1}_{\mathrm{Mod}(\mathcal{C}/\mathcal{I})}\Big(\mathrm{Hom}_{\mathcal{C}/\mathcal{I}}(C,-), F'\Big)=
\mathrm{Ext}^{1}_{\mathrm{Mod}(\mathcal{C})}\Big(\frac{\mathrm{Hom}_{\mathcal{C}}(C,-)}{\mathcal{I}(C,-)}, F'\circ \pi\Big)$$ for all $F'\in \mathrm{Mod}(\mathcal{C}/\mathcal{I})$ and for all $C\in \mathcal{C}$.
Consider the following exact sequence $\xymatrix{0\ar[r] & \mathcal{I}\Big(C,-\Big)\ar[r] & \mathcal{C}\Big(C,-\Big)\ar[r]^{\Psi} & H\Big(C,-\Big)\ar[r] & 0,}$
with $\mathcal{I}(C,-)$ and $\mathcal{C}(C,-)$ projective in $\mathrm{Mod}(\mathcal{C})$. Thus, the projective dimension of each $\frac{\mathrm{Hom}_{\mathcal{C}}(C,-)}{\mathcal{I}(C,-)}$ is less than or equal to $1$. Therefore, we have that $\mathrm{Ext}^{j}_{\mathrm{Mod}(\mathcal{C})}\Big(\frac{\mathrm{Hom}_{\mathcal{C}}(C,-)}{\mathcal{I}(C,-)}, F'\circ \pi\Big)=0$ for all $F'\in \mathrm{Mod}(\mathcal{C}/\mathcal{I})$, for all $C\in \mathcal{C}$ and for all $j\geq 2$. Moreover, by Proposition \ref{caractidem2}, we have that $\mathcal{I}$ is strongly idempotent.
\end{proof}

\begin{proposition}\label{prelimiextvannew}
Let $\mathcal{I}$ be an idempotent ideal such that $\mathcal{I}(C,-)$ is projective in $\mathrm{Mod}(\mathcal{C})$ for all $C\in \mathcal{C}$. Consider the canonical projection $\Phi:\mathcal{C}\longrightarrow \mathcal{C}/\mathcal{I}=\mathcal{B}$  and $H:=\mathcal{B}(-,-)\circ (\Phi^{op}\otimes \Phi)$.  Then
\begin{enumerate}
\item [(a)] $\mathrm{Ext}^{i}_{\mathrm{Mod}(\mathcal{C}^{e})}(\mathcal{I},H)=0$ for all $i>0$.

\item [(b)] $\mathrm{Ext}^{i}_{\mathrm{Mod}(\mathcal{C}^{e})}(\mathcal{C},H)\simeq \mathrm{Ext}^{i}_{\mathrm{Mod}(\mathcal{C}^{e})}(H,H)$  for all $i\geq 1$.

\end{enumerate}
\end{proposition}
\begin{proof}
$(a)$. Consider the standard projective resolution  $\Big(\mathbb{S}(\mathcal{C}),d\Big)$ of $\mathcal{C}$.  By the isomorphism given in Equation \ref{certainadjun}, we have an isomorphism of complexes
$$\mathrm{Hom}_{\mathrm{Mod}(\mathcal{C}^{op}\otimes_{K}\mathcal{C})}\Big(\mathcal{I}\boxtimes_{\mathcal{C}^{op}}\mathbb{S}(\mathcal{C}),H\Big)\simeq  \mathrm{Hom}_{\mathrm{Mod}(\mathcal{C}^{op}\otimes_{K} \mathcal{C})}\Big(\mathbb{S}(\mathcal{C}), \mathrm{Hom}_{\mathrm{Mod}(\mathcal{C})}(\mathcal{I},H)\Big).$$
By Proposition \ref{lemaduda} we have that $\mathcal{I}\boxtimes_{\mathcal{C}^{op}}\mathbb{S}(\mathcal{C})$ is a projective resolution of $\mathcal{I}$ in $\mathrm{Mod}(\mathcal{C}^{e})$. By Proposition \ref{proyepihomo} we get that $\mathcal{I}$ is strongly idempotent, and hence by Corollary \ref{HomIH=0}(b), we obtain that $\mathrm{Hom}_{\mathrm{Mod}(\mathcal{C})}(\mathcal{I},H)=0$. By taking homology to the isomorphism of complexes above we obtain 
\begin{align*}
\mathrm{Ext}^{i}_{\mathrm{Mod}(\mathcal{C}^{e})}(\mathcal{I},H) & =H^{i}\Big(\mathrm{Hom}_{\mathrm{Mod}(\mathcal{C}^{op}\otimes_{K}\mathcal{C})}\Big(\mathcal{I}\boxtimes_{\mathcal{C}^{op}}\mathbb{S}(\mathcal{C}),H\Big)\Big)\\
& \simeq H^{i}\Big(\mathrm{Hom}_{\mathrm{Mod}(\mathcal{C}^{op}\otimes_{K} \mathcal{C})}\Big(\mathbb{S}(\mathcal{C}), \mathrm{Hom}_{\mathrm{Mod}(\mathcal{C})}(\mathcal{I},H)\Big)\Big)=0.
\end{align*}
$(b)$. By Corollary \ref{HomIH=0} we have that  $\mathrm{Hom}_{\mathrm{Mod}(\mathcal{C}^{e})}(\mathcal{I},H)=0$, and by item $(a)$ we have that $\mathrm{Ext}^{i}_{\mathrm{Mod}(\mathcal{C}^{e})}(\mathcal{I},H)=0$
for all $i>0$. Then, by the long exact sequence  obtained by applying the functor $\mathrm{Hom}_{\mathrm{Mod}(\mathcal{C}^{e})}(-,H)$ to the exact sequence given in Equation \ref{succanonica}, we get $\mathrm{Ext}^{i}_{\mathrm{Mod}(\mathcal{C}^{e})}(\mathcal{C},H)\simeq \mathrm{Ext}^{i}_{\mathrm{Mod}(\mathcal{C}^{e})}(H,H)$ for all $i\geq 1$.
\end{proof}

\begin{proposition} \label{otroisoimportantenew}
Let $\mathcal{I}$ be a strongly idempotent ideal of $\mathcal{C}$. Then 
$\mathrm{Ext}_{\mathrm{Mod}(\mathcal{C}^{e})}^{i}(H,H)$ $\simeq \mathrm{Ext}_{\mathrm{Mod}(\mathcal{B}^{e})}^{i}(\mathcal{B},\mathcal{B})=H^{i}(\mathcal{B})$ for all $i\geq 1$.
\end{proposition}
\begin{proof}
By Proposition \ref{piehomoloepi}, we conclude that $\Phi^{op}\otimes\Phi:\mathcal{C}^{e}\longrightarrow \mathcal{B}^{e}$ is a homological epimorphism (that is, $\mathrm{Ker}(\Phi^{op}\otimes\Phi)$ is a strongly idempotent ideal of $\mathcal{C}^{e}$). 
We have the functor $(\Phi^{op}\otimes\Phi)_{\ast}:\mathrm{Mod}(\mathcal{B}^{e})\longrightarrow \mathrm{Mod}(\mathcal{C}^{e}).$ Hence by  Corollary \ref{caractidem2}, we have that
$$\varphi^{i}_{F,(\Phi^{op}\otimes\Phi)_{\ast}(F')}:\mathrm{Ext}^{i}_{\mathrm{Mod}(\mathcal{B}^{e})}(F,F')\longrightarrow \mathrm{Ext}^{i}_{\mathrm{Mod}(\mathcal{C}^{e})}((\Phi^{op}\otimes\Phi)_{\ast}(F),(\Phi^{op}\otimes\Phi)_{\ast}(F'))$$ is an isomorphism for all $F,F'\in \mathrm{Mod}(\mathcal{B}^{e})$ and for all  $0\leq i<\infty$. In particular for
$F=F'=\mathcal{B}$ we have an isomorphism
$$\mathrm{Ext}^{i}_{\mathrm{Mod}(\mathcal{B}^{e})}(\mathcal{B},\mathcal{B})\simeq \mathrm{Ext}^{i}_{\mathrm{Mod}(\mathcal{C}^{e})}((\Phi^{op}\otimes\Phi)_{\ast}(\mathcal{B}),(\Phi^{op}\otimes\Phi)_{\ast}(\mathcal{B}))=\mathrm{Ext}^{i}_{\mathrm{Mod}(\mathcal{C}^{e})}(H,H).$$
\end{proof}

The following is a generalization of the first exact sequence obtained in \cite[Theorem 3.4 (1)]{KoenigNagase}.

\begin{theorem}\label{episuclarga}
Let $\mathcal{I}$ be an idempotent ideal such that $\mathcal{I}(C,-)$ is projective in $\mathrm{Mod}(\mathcal{C})$ for all $C\in \mathcal{C}$. Then $\Phi:\mathcal{C}\longrightarrow \mathcal{B}=\mathcal{C}/\mathcal{I}$ a homological epimorphism, and there is a long exact sequence that relates the Hochschild-Mitchell cohomology of $\mathcal{C}$ to the Hochschild-Mitchell cohomology of $\mathcal{B}=\mathcal{C}/\mathcal{I}$ 
$$\xymatrix{0\ar[r] & \mathrm{Hom}_{\mathrm{Mod}(\mathcal{C}^{e})}(\mathcal{C},\mathcal{I})\ar[r] &   H^{0}(\mathcal{C})\ar[r] &  H^{0}(\mathcal{B})\ar `[ld] `[] `[llld] |{\delta} `[] [llldr] & \\
&  \mathrm{Ext}_{\mathrm{Mod}(\mathcal{C}^{e})}^{1}(\mathcal{C},\mathcal{I})\ar[r] &   H^{1}(\mathcal{C})\ar[r] &   H^{1}(\mathcal{B})\ar `[ld] `[] `[llld] |{\delta} `[] [llldr] &\\
&  \mathrm{Ext}_{\mathrm{Mod}(\mathcal{C}^{e})}^{2}(\mathcal{C},\mathcal{I})\ar[r] &   H^{2}(\mathcal{C})\ar[r] &   H^{2}(\mathcal{B})\ar[r] & \cdots}$$
\end{theorem}
\begin{proof}
By Proposition \ref{proyepihomo}, we know that $\Phi:\mathcal{C}\longrightarrow \mathcal{C}/\mathcal{I}$ is a homological epimorphism. Let us consider the exact sequence
$$\xymatrix{0\ar[r] & \mathcal{I}\ar[r] & \mathcal{C}\ar[rr]^(.5){\Gamma(\Phi)} &  & H\ar[r] & 0}$$
By applying the functor $\mathrm{Hom}_{\mathrm{Mod}(\mathcal{C}^{e})}(\mathcal{C},-)$ to the last exact sequence, we have that
$$\xymatrix{0\ar[r] & \mathrm{Hom}_{\mathrm{Mod}(\mathcal{C}^{e})}(\mathcal{C},\mathcal{I})\ar[r] &   \mathrm{Hom}_{\mathrm{Mod}(\mathcal{C}^{e})}(\mathcal{C},\mathcal{C})\ar[r] &  \mathrm{Hom}_{\mathrm{Mod}(\mathcal{C}^{e})}(\mathcal{C},H) \ar `[ld] `[] `[llld] |{\delta} `[] [llldr] & \\
&  \mathrm{Ext}_{\mathrm{Mod}(\mathcal{C}^{e})}^{1}(\mathcal{C},\mathcal{I})\ar[r] &    \mathrm{Ext}_{\mathrm{Mod}(\mathcal{C}^{e})}^{1}(\mathcal{C},\mathcal{C})\ar[r] &   \mathrm{Ext}_{\mathrm{Mod}(\mathcal{C}^{e})}^{1}(\mathcal{C},H)\ar `[ld] `[] `[llld] |{\delta} `[] [llldr] &\\
&  \mathrm{Ext}_{\mathrm{Mod}(\mathcal{C}^{e})}^{2}(\mathcal{C},\mathcal{I})\ar[r] &    \mathrm{Ext}_{\mathrm{Mod}(\mathcal{C}^{e})}^{2}(\mathcal{C},\mathcal{C})\ar[r] &   \mathrm{Ext}_{\mathrm{Mod}(\mathcal{C}^{e})}^{2}(\mathcal{C},H) \cdots}$$

By Proposition \ref{propopartida1new}, we have that  $\mathrm{Hom}_{\mathrm{Mod}(\mathcal{C}^{e})}(\mathcal{C},H)\simeq \mathrm{Hom}_{\mathrm{Mod}(\mathcal{C}^{e})}(H,H)\simeq H^{0}(\mathcal{B})$, and by Proposition \ref{prelimiextvannew} and \ref{otroisoimportantenew} we get the isomorphisms  $\mathrm{Ext}^{i}_{\mathrm{Mod}(\mathcal{C}^{e})}(\mathcal{C},H)\simeq \mathrm{Ext}^{i}_{\mathrm{Mod}(\mathcal{C}^{e})}(H,H)\simeq H^{i}(\mathcal{B})$ for all $i\geq 1$, which proves the theorem.

\end{proof}

\section{Applications}
We consider the triangular matrix category  $\Lambda:=\left[\begin{smallmatrix}
\mathcal{T} & 0 \\ 
M & \mathcal{U}
\end{smallmatrix}\right]$ constructed in \cite{LeOS} and defined as follows.
\begin{definition}$\textnormal{\cite[Definition 3.5]{LeOS}}$ \label{defitrinagularmat}
Let $\mathcal{U}$ and $\mathcal{T}$ be two $K$-categories, and consider an additive $K$-functor  $M$ from the tensor product category  $\mathcal{U}\otimes_{K} \mathcal{T}^{op}$ to the category $\mathrm{Mod}(K)$.
The \textbf{triangular matrix category}
$\Lambda=\left[ \begin{smallmatrix}
\mathcal{T} & 0 \\ M & \mathcal{U}
\end{smallmatrix}\right]$ is defined as below.
\begin{enumerate}
\item [(a)] The class of objects of this category are matrices $ \left[
\begin{smallmatrix}
T & 0 \\ M & U
\end{smallmatrix}\right]  $ with $ T\in \mathrm{obj}(\mathcal{T}) $ and $ U\in \mathrm{obj}(\mathcal{U}) $.

\item [(b)] For objects in
$\left[ \begin{smallmatrix}
T & 0 \\
M & U
\end{smallmatrix} \right] ,  \left[ \begin{smallmatrix}
T' & 0 \\
M & U'
\end{smallmatrix} \right]$ in
$\Lambda$, we define $\mathrm{ Hom}_{\Lambda}\left (\left[ \begin{smallmatrix}
T & 0 \\
M & U
\end{smallmatrix} \right] ,  \left[ \begin{smallmatrix}
T' & 0 \\
M & U'
\end{smallmatrix} \right]  \right)  := \left[ \begin{smallmatrix}
\mathrm{Hom}_{\mathcal{T}}(T,T') & 0 \\
M(U',T) & \mathrm{Hom}_{\mathcal{U}}(U,U')
\end{smallmatrix} \right].$
\end{enumerate}
The composition is given by
\begin{eqnarray*}
\circ&:&\left[  \begin{smallmatrix}
{\mathcal{T}}(T',T'') & 0 \\
M(U'',T') & {\mathcal{U}}(U',U'')
\end{smallmatrix}  \right] \times \left[
\begin{smallmatrix}
{\mathcal{T}}(T,T') & 0 \\
M(U',T) & {\mathcal{U}}(U,U')
\end{smallmatrix} \right]\longrightarrow\left[
\begin{smallmatrix}
{\mathcal{T}}(T,T'') & 0 \\
M(U'',T) & {\mathcal{U}}(U,U'')\end{smallmatrix} \right] \\
&& \left( \left[ \begin{smallmatrix}
t_{2} & 0 \\
m_{2} & u_{2}
\end{smallmatrix} \right], \left[
\begin{smallmatrix}
t_{1} & 0 \\
m_{1} & u_{1}
\end{smallmatrix} \right]\right)\longmapsto\left[
\begin{smallmatrix}
t_{2}\circ t_{1} & 0 \\
m_{2}\bullet t_{1}+u_{2}\bullet m_{1} & u_{2}\circ u_{1}
\end{smallmatrix} \right].
\end{eqnarray*}
\end{definition}
We recall that $ m_{2}\bullet t_{1}:=M(1_{U''}\otimes t_{1}^{op})(m_{2})$ and
$u_{2}\bullet m_{1}=M(u_{2}\otimes 1_{T})(m_{1})$.
Thus, $\Lambda$ is clearly a $K$-category since
$\mathcal{T} $ and $\mathcal{U}$ are $K$-categories and $M(U',T)$ is a $K$-module.\\
We define a functor $\Phi:\Lambda\longrightarrow \mathcal{U}$
as follows: $\Phi\Big(\left[\begin{smallmatrix}
 T & 0 \\ 
M & U
\end{smallmatrix}\right]\Big):=U$ and for 
 $\left[\begin{smallmatrix}
\alpha & 0 \\ 
m & \beta
\end{smallmatrix}\right]:\left[\begin{smallmatrix}
 T & 0 \\ 
M & U
\end{smallmatrix}\right]\longrightarrow \left[\begin{smallmatrix}
T' & 0 \\ 
M & U'
\end{smallmatrix}\right]$ we set $\Phi\Big(\left[\begin{smallmatrix}
\alpha & 0 \\ 
m & \beta
\end{smallmatrix}\right]\Big)=\beta$.\\
For simplicity, we will write
$\mathfrak{M}=\left[\begin{smallmatrix}
 T & 0 \\ 
M & U
\end{smallmatrix}\right]\in \Lambda$.

\begin{lemma}\label{exacseuqnimpor}
 There exists an exact sequence
in $\mathrm{Mod}(\Lambda^{e})$
$$\xymatrix{0\ar[r] & \mathcal{I}\ar[r] & \Lambda\ar[rr]^(.3){\Gamma(\Phi)} & & \mathcal{U}(-,-)\circ (\Phi_{2}\otimes\Phi_{1})\ar[r] & 0,}$$
where for objects $\mathfrak{M}'=\left[\begin{smallmatrix}
 T' & 0 \\ 
M & U'
\end{smallmatrix}\right]$ and $\mathfrak{M}=\left[\begin{smallmatrix}
 T & 0 \\ 
M & U
\end{smallmatrix}\right]$ in $\Lambda$ the ideal $\mathcal{I}$ is given as  $\mathcal{I}\big(\mathfrak{M}',\mathfrak{M}\big)=\mathrm{Ker}\left([\Gamma (\Phi)]_{\big(\mathfrak{M}',\mathfrak{M}\big)}\right)=\left[\begin{smallmatrix}
 \mathcal{T}(T',T) & 0 \\ 
M(U,T') &  0
\end{smallmatrix}\right]$.
\end{lemma}
\begin{proof}
It is straightforward.
\end{proof}



\begin{remark}\label{projectivida}
We can see that $\mathcal{I}\Big({\left[\begin{smallmatrix}
 T & 0 \\ 
M & U
\end{smallmatrix}\right]},-\Big)\simeq {\Lambda}\Big({\left[\begin{smallmatrix}
 T & 0 \\ 
M & 0
\end{smallmatrix}\right]},-\Big)$, and, from this, it follows that $\mathcal{I}\Big({\left[\begin{smallmatrix}
 T & 0 \\ 
M & U
\end{smallmatrix}\right]},-\Big)$ is projective in $\mathrm{Mod}(\Lambda)$.
\end{remark}

The following extends the well-known result discovered independently by Cibils and Michelena-Platzeck; see  \cite{Cibils} and \cite{MichelanaPlat}.

\begin{theorem}\label{Michelena}(Cibils-Michelena-Platzeck's long exact sequence)
Let $\Lambda=\left[ \begin{smallmatrix}
\mathcal{T} & 0 \\ M & \mathcal{U}
\end{smallmatrix}\right]$ be a triangular matrix category. Then there is a long exact sequence that relates the Hochschild-Mitchell cohomology of $\Lambda$ to the Hochschild-Mitchell cohomology of $\mathcal{U}$: 
$$\xymatrix{0\ar[r] & \mathrm{Hom}_{\mathrm{Mod}(\Lambda^{e})}(\Lambda,\mathcal{I})\ar[r] &   H^{0}(\Lambda)\ar[r] &  H^{0}(\mathcal{U})\ar `[ld] `[] `[llld] |{\delta} `[] [llldr] & \\
&  \mathrm{Ext}_{\mathrm{Mod}(\Lambda^{e})}^{1}(\Lambda,\mathcal{I})\ar[r] &   H^{1}(\Lambda)\ar[r] &   H^{1}(\mathcal{U})\ar `[ld] `[] `[llld] |{\delta} `[] [llldr] &\\
&  \mathrm{Ext}_{\mathrm{Mod}(\Lambda^{e})}^{2}(\Lambda,\mathcal{I})\ar[r] &   H^{2}(\Lambda)\ar[r] &   H^{2}(\mathcal{U})\ar[r] & \cdots}$$
\end{theorem}
\begin{proof}
We have an epimorphism $\Phi:\Lambda\longrightarrow \mathcal{U}$ and an exact sequence
in $\mathrm{Mod}(\Lambda^{e})$
$$\xymatrix{0\ar[r] & \mathcal{I}\ar[r] & \Lambda\ar[rr]^(.3){\Gamma(\Phi)} & & \mathcal{U}(-,-)\circ (\Phi^{op}\otimes\Phi)\ar[r] & 0.}$$
We notice that $\mathcal{I}$ is an ideal of $\Lambda$ and $\mathcal{U}\simeq \Lambda/\mathcal{I}$.
By Remark  \ref{projectivida}, we get that $\mathcal{I}(\mathfrak{M},-)$ is projective in $\mathrm{Mod}(\Lambda)$ for all $\mathfrak{M}\in \Lambda$.\\
Now, for an object  $\Big(\left[\begin{smallmatrix}
 T' & 0 \\ 
M & U'
\end{smallmatrix}\right],\left[\begin{smallmatrix}
 T & 0 \\ 
M & U
\end{smallmatrix}\right]\Big)$ in $\Lambda^{op}\otimes_{K}\Lambda=\Lambda^{e}$,  we obtain that

$$\mathcal{I}\Big(\left[\begin{smallmatrix}
 T' & 0 \\ 
M & U'
\end{smallmatrix}\right],\left[\begin{smallmatrix}
 T & 0 \\ 
M & U
\end{smallmatrix}\right]\Big)=\left[\begin{smallmatrix}
 \mathcal{T}(T',T) & 0 \\ 
M(U,T') &  0
\end{smallmatrix}\right]\subseteq \left[\begin{smallmatrix}
\mathcal{T}(T',T) & 0 \\ 
M(U,T') & \mathcal{U}(U',U) 
\end{smallmatrix}\right].$$
We assert that $\mathcal{I}$ is idempotent. Indeed, for $\left[\begin{smallmatrix}
f & 0 \\ 
m & 0
\end{smallmatrix}\right]\in \mathcal{I}\Big(\left[\begin{smallmatrix}
 T' & 0 \\ 
M & U'
\end{smallmatrix}\right],\left[\begin{smallmatrix}
 T & 0 \\ 
M & U
\end{smallmatrix}\right]\Big)$, we get
$\left[\begin{smallmatrix}
1_{T'} & 0 \\ 
0 & 0
\end{smallmatrix}\right]\in \mathcal{I}\Big(\left[\begin{smallmatrix}
 T' & 0 \\ 
M & U'
\end{smallmatrix}\right],\left[\begin{smallmatrix}
 T' & 0 \\ 
M & U'
\end{smallmatrix}\right]\Big)=\left[\begin{smallmatrix}
 \mathcal{T}(T',T') & 0 \\ 
M(U',T') &  0
\end{smallmatrix}\right]$ and

$$\left[\begin{smallmatrix}
f & 0 \\ 
m & 0
\end{smallmatrix}\right]\circ \left[\begin{smallmatrix}
1_{T'} & 0 \\ 
0 & 0
\end{smallmatrix}\right]=\left[\begin{smallmatrix}
f\circ 1_{T'} & 0 \\ 
m\bullet 1_{T'}& 0
\end{smallmatrix}\right]=
\left[\begin{smallmatrix}
f & 0 \\ 
m & 0
\end{smallmatrix}\right]$$
since $ m \bullet 1_{T'}:=M(1_{U'}\otimes 1_{T'}^{op})(m)=m$  because $M(1_{U'}\otimes 1_{T'}^{op})=1_{M(U',T')}$.
This proves that $\mathcal{I}^{2}=\mathcal{I}$, and hence $\mathcal{I}$ is an idempotent ideal of $\Lambda$. Therefore, by Theorem \ref{episuclarga}, we have required exact sequence.
\end{proof}

\subsection{Happel's exact sequence}

In this section, $\mathcal{U}$ will denote a $K$-category and $M:\mathcal{U}\longrightarrow \mathrm{Mod}(K)$ a $K$-functor. We consider $\mathcal{C}_{K}$ the $K$-category with only one object, namely $\mathrm{obj}(\mathcal{C}_{K}):=\{\ast\}$, and the isomorphism  $\Delta: \mathcal{U}\otimes \mathcal{C}_{K}^{op}\longrightarrow \mathcal{U}$ given before Definition \ref{prodtensorconcat}.
We then get  $\underline{M}:\mathcal{U}\otimes \mathcal{C}_{K}^{op}\longrightarrow \mathrm{Mod}(K)$ given as
$\underline{M}:=\Delta\circ M$.
Hence, we can construct the matrix category $\Lambda:=\left[ \begin{smallmatrix}
\mathcal{C}_{K} & 0 \\ 
\underline{M} & \mathcal{U}
\end{smallmatrix}\right].$ This matrix category is called the $\textbf{one-point extension category}$ because it is a generalization of the well-known construction of  one point-extension algebra; see for example page 71 in \cite{AusBook}. In this case the exact sequence given in Theorem \ref{Michelena} has another form that reduces to a long exact sequence that is a generalization of the one given by D. Happel in \cite[Theorem 5.3]{Happel} on page 12; see also article \cite{EduardoGreen}.

The following proposition is a generalization of Theorem 2.8a on page 167 in \cite{Cartan}. 

\begin{proposition}\label{formula2.8aEilenberg}
Let $\mathcal{C},\mathcal{D},\mathcal{E}$ be $K$-projective $K$-categories. Consider functors $F\in  \mathrm{Fun}_{K}(\mathcal{C},\mathrm{Fun}_{K}(\mathcal{E}^{op},{\mathrm{Mod}(K)}))$,  $G\in \mathrm{Fun}_{K}(\mathcal{C}^{op}\otimes_{K} \mathcal{D},{\mathrm{Mod}(K)})$ and  $H\in \mathrm{Fun}_{K}(\mathcal{D},\mathrm{Fun}_{K}(\mathcal{E}^{op},{\mathrm{Mod}(K)}))$. Suppose that
$\mathrm{Tor}_{n}^{\mathcal{C}}(F,G(-,D))=0$  for all  $D\in \mathcal{D}$  and $\forall \,n>0$, and also that $\mathrm{Ext}^{n}_{\mathcal{D}}(G(C,-),H)=0$ for all $C\in \mathcal{C}^{op}$ and $\forall \, n>0$. Then, there exists an isomorphism for all $i\geq 0$:
\begin{equation}
\mathrm{Ext}^{i}_{\mathrm{Mod}(\mathcal{D}\otimes_{K} \mathcal{E}^{op})}\Big(F\boxtimes_{\mathcal{C}}G,H\Big)\simeq \mathrm{Ext}^{i}_{\mathrm{Mod}(\mathcal{E}^{op}\otimes_{K} \mathcal{C})}\Big(F,\mathbb{HOM}_{\mathcal{D}}(G,H)\Big),
\end{equation}
where $\mathbb{HOM}_{\mathcal{D}}(G,H)$ denotes the symbolic hom defined on page 30 in \cite{Mitchelring}.
\end{proposition}
\begin{proof}
The proof is similar to that of Proposition \ref{isocartan}.
\end{proof}

\begin{corollary}(Happel's long exact sequence)\label{longHappel}
Let $M:\mathcal{U}\longrightarrow \mathrm{Mod}(K)$ be a $K$-functor. Consider the one point extension category
$\Lambda:=\left[ \begin{smallmatrix}
\mathcal{C}_{K} & 0 \\ 
\underline{M} & \mathcal{U}
\end{smallmatrix}\right].$
Then there exists a long exact sequence:
$$\xymatrix{0\ar[r] &  H^{0}(\Lambda)\ar[r] &   H^{0}(\mathcal{U})\ar[r] & \mathrm{Hom}_{\mathrm{Mod}(\mathcal{U})}(M,M)/K\ar `[ld] `[] `[llld] |{} `[] [llldr] & \\
&  H^{1}(\Lambda)\ar[r] &   H^{1}(\mathcal{U})\ar[r] &   \mathrm{Ext}_{\mathrm{Mod}(\mathcal{U})}^{1}(M,M)\ar[r] & \cdots &}$$
\end{corollary}
\begin{proof}
By following ideas in \cite{Redondo} on page 133 and using Proposition \ref{formula2.8aEilenberg} we can show that  $\mathrm{Hom}_{\mathrm{Mod}(\Lambda^{e})}(\Lambda,\mathcal{I})=0$ and  $\mathrm{Ext}_{\mathrm{Mod}(\Lambda^{e})}^{1}(\Lambda,\mathcal{I})= \mathrm{Hom}_{\mathrm{Mod}(\mathcal{U})}(M,M)/K$ and that $\mathrm{Ext}_{\mathrm{Mod}(\Lambda^{e})}^{n}(\Lambda,\mathcal{I})\simeq  \mathrm{Ext}_{\mathrm{Mod}(\mathcal{U})}^{n-1}(M,M)$ for all $n\geq 2$. The result follows from Theorem \ref{Michelena}.
\end{proof}

\subsection{Recollements and torsion pairs}

Let $\mathcal{B}$ be a full additive subcategory of $\mathcal{C}$.  Given $C, C'\in \mathcal{C}$ we denote by $\mathcal{I}_{\mathcal{B}}(C,C')$ the subset of $\mathcal{C}(C,C')$ consisting of morphisms which factor through some object in $\mathcal{B}$. This defines the two-sided ideal $\mathcal{I}_{\mathcal{B}}$ which is an idempotent ideal in $\mathcal{C}$.\\
A pair $(\mathcal{T},\mathcal{F})$ of full subcategories of $\mathcal{C}$ is a torsion pair if the following conditions hold.

\begin{enumerate}
\item [(a)] $\mathrm{Hom}_{\mathcal{C}}(M,N)=0$ for all $M\in \mathcal{T}, N\in \mathcal{F}$.

\item [(b)] For all $C\in \mathcal{C}$ there exists an exact sequence
$$\xymatrix{0\ar[r] & Y_{C}\ar[r] & C\ar[r] & Z^{C}\ar[r] & 0}$$
with $Y_{C}\in \mathcal{T}$ and $Z^{C}\in \mathcal{F}$.
\end{enumerate}
In this case $\mathcal{T}$ is called a torsion class and $\mathcal{F}$ a torsion free class. It is well known that $\mathcal{T}$ is a torsion class if and only if $\mathcal{T}$ is closed under quotients,  coproducts and extensions. Dually, $\mathcal{F}$ is a torsion free class if and only if $\mathcal{F}$ is closed under subobjects,  products and extensions.  A triple $(\mathcal{X},\mathcal{Y},\mathcal{Z})$ of full subcategories of $\mathcal{C}$ is a TTF triple  if $(\mathcal{X},\mathcal{Y})$ and $(\mathcal{Y},\mathcal{Z})$ are torsion pairs.\\
Let $\mathcal{C}$ be a $K$-category such that $\mathrm{Mod}(\mathcal{C})$ has global dimension equal to $1$, that is,  $\mathrm{Mod}(\mathcal{C})$  is hereditary. We have the following result.

\begin{proposition}\label{homottf}
Let  $\mathcal{C}$ be a $K$-category such that $\mathrm{Mod}(\mathcal{C})$  is hereditary. There exists a bijection between the class of  TTF triples  $(\mathcal{T},\mathcal{F},\mathcal{F}')$ in $\mathrm{Mod}(\mathcal{C})$ and homological epimorphisms $\pi:\mathcal{C}\longrightarrow \mathcal{C}/\mathcal{I}$.
\end{proposition}
\begin{proof}
Let $(\mathcal{T},\mathcal{F},\mathcal{F}')$ be a TTF triple, we consider the idempotent ideal $\mathcal{I}(A,B)=\{f\mid M(f)=0,\,\,\forall M\in F\}$. Hence, we get the functor $\pi:\mathcal{C}\longrightarrow \mathcal{C}/\mathcal{I}$ and we obtain that $\mathcal{F}\simeq \mathrm{Mod}(\mathcal{C}/\mathcal{I})$. Since $\mathcal{I}$ is idempotent, we have that
there exists isomorphisms
$$\mathrm{Ext}^{i}_{\mathrm{Mod}(\mathcal{C}/\mathcal{I})}(M,N)\simeq 
\mathrm{Ext}^{i}_{\mathrm{Mod}(\mathcal{C})}(M\circ \pi,N\circ \pi)$$
for $i=0,1$. Since $\mathrm{Mod}(\mathcal{C})$ is hereditary, we get that
$\mathrm{Ext}^{1}_{\mathrm{Mod}(\mathcal{C})}(M',-)$ is exact for all $M'\in \mathrm{Mod}(\mathcal{C})$. Hence, we have that $\mathrm{Ext}^{1}_{\mathrm{Mod}(\mathcal{C}/\mathcal{I})}(M,-)$ is exact for all $M\in \mathrm{Mod}(\mathcal{C}/\mathcal{I})$ and, thus $\mathrm{Mod}(\mathcal{C}/\mathcal{I})$ is hereditary. Therefore,
$$\mathrm{Ext}^{i}_{\mathrm{Mod}(\mathcal{C}/\mathcal{I})}(M,N)\simeq 
\mathrm{Ext}^{i}_{\mathrm{Mod}(\mathcal{C})}(M\circ \pi,N\circ \pi)=0$$
for all $i\geq 2$. Proving that $\phi:\mathcal{C}\longrightarrow \mathcal{C}/\mathcal{I}$ is a homological epimorphism.\\
Now, let $\phi:\mathcal{C}\longrightarrow \mathcal{C}/\mathcal{I}$ be a homological epimorphism, the associated TTF is
$$\Big(\,^{\perp}\mathrm{Mod}(\mathcal{C}/\mathcal{I}),\mathrm{Mod}(\mathcal{C}/\mathcal{I}),\mathrm{Mod}(\mathcal{C}/\mathcal{I})^{\perp}\Big).$$
It is easy to see that this assignments are bijective and inverse of each other.
\end{proof}

\begin{proposition}\label{homotorsion}
Let $\mathcal{C}$ be a complete and cocomplete abelian $K$-category and  $\mathcal{F}$ a torsion free class. Hence, $\pi:\mathcal{C}\longrightarrow \mathcal{C}/\mathcal{I}_{\mathcal{F}}$ is a homological epimorphism.
\end{proposition}
\begin{proof}
Consider the  torsion pair $(\mathcal{T},\mathcal{F})$ associated to $\mathcal{F}$. Hence, for $C\in \mathcal{C}$ there exists a unique exact sequence
$$\xymatrix{0\ar[r] & Y_{C}\ar[r]^{i} & C\ar[r]^{p} & Z^{C}\ar[r] & 0}$$
with $Y_{C}\in \mathcal{T}$ and $Z^{C}\in \mathcal{F}$.\\
Now, let $\alpha:C\longrightarrow Z$ with $Z\in \mathcal{F}$. Since $(\mathcal{T},\mathcal{F})$ is a torsion pair and $Y_{C}\in \mathcal{T}$, we have that $\mathcal{C}(Y_{C},Z)=0$, and hence $\alpha i=0$. Thus, by the universal property of the cokernel, there exists a unique morphism $\alpha':Z^{C}\longrightarrow Z$ such that $\alpha=\alpha' p$. That is, the following diagram commutes
$$\xymatrix{0\ar[r] & Y_{C}\ar[r]^{i} & C\ar[r]^{p}\ar[dr]_{\alpha} & Z^{C}\ar[r]\ar[d]^{\alpha'} & 0\\
 & & & Z.}$$
For each $C'\in \mathcal{C}$ we define the morphism
$$\Psi_{C'}:\mathcal{C}(Z^{C},C')\longrightarrow \mathcal{I}_{\mathcal{F}}(C,C'),$$
given as $\Psi_{C'}(\gamma)=\gamma p$. Since $p$ is an epimorphism, we have that $\Psi_{C'}$ is injective.\\
Now, let $f\in \mathcal{I}_{\mathcal{F}}(C,C')$. Then, there exists $f_{1}:C\longrightarrow Z$ and $f_{2}:Z\longrightarrow C'$ such that $f=f_{2}f_{1}$ and $Z\in \mathcal{F}$. By the discussion above, there exists $f_{1}':Z^{C}\longrightarrow Z$ such that $f_{1}=f_{1}' p.$ Hence, $f=(f_{2}f_{1}')p=\Psi_{C'}(f_{2}f_{1}')$, proving that  $\Psi_{C'}$ is surjective and, we conclude that
$$\Psi=\mathcal{C}(p,-):\mathcal{C}(Z^{C},-)\longrightarrow \mathcal{I}_{\mathcal{F}}(C,-)$$
is an isomorphism. Thus, $\mathcal{I}_{\mathcal{F}}(C,-)$ is  a projective $\mathcal{C}$-module for all $C\in \mathcal{C}$. By Proposition \ref{proyepihomo}, we get that $\pi:\mathcal{C}\longrightarrow \mathcal{C}/\mathcal{I}_{\mathcal{F}}$ is a homological epimorphism. 
\end{proof}

\begin{corollary}\label{lonseqiencetor}
Let $\mathcal{C}$ be a complete and cocomplete abelian $K$-category and $(\mathcal{T},\mathcal{F})$ a torsion theory. Then, there is a long exact sequence that relates the Hochschild-Mitchell cohomology of $\mathcal{C}$ to the Hochschild-Mitchell cohomology of $\mathcal{B}=\mathcal{C}/\mathcal{I}_{\mathcal{F}}$ 
$$\xymatrix{0\ar[r] & \mathrm{Hom}_{\mathrm{Mod}(\mathcal{C}^{e})}(\mathcal{C},\mathcal{I}_{\mathcal{F}})\ar[r] &   H^{0}(\mathcal{C})\ar[r] &  H^{0}(\mathcal{B})\ar `[ld] `[] `[llld] |{\delta} `[] [llldr] & \\
&  \mathrm{Ext}_{\mathrm{Mod}(\mathcal{C}^{e})}^{1}(\mathcal{C},\mathcal{I}_{\mathcal{F}})\ar[r] &   H^{1}(\mathcal{C})\ar[r] &   H^{1}(\mathcal{B})\ar `[ld] `[] `[llld] |{\delta} `[] [llldr] &\\
&  \mathrm{Ext}_{\mathrm{Mod}(\mathcal{C}^{e})}^{2}(\mathcal{C},\mathcal{I}_{\mathcal{F}})\ar[r] &   H^{2}(\mathcal{C})\ar[r] &   H^{2}(\mathcal{B})\ar[r] & \cdots}$$
\end{corollary}
\begin{proof}
It follows by Theorem \ref{episuclarga} and Proposition \ref{homotorsion}.
\end{proof}

\vspace{0.5cm}
Let us consider an example. Let $A$ be the quotient path $K$-algebra given by the quiver
$$\xymatrix{ & 2\ar[dr]^{\alpha} & \\
1\ar[ur]^{\alpha} & & 3\ar[ll]_{\alpha}}$$
and the third power of the ideal generated by all the arrows. The Auslander-Reiten quiver can be seen in the Figure \ref{ARquiver}, where every module is represented by its Loewy series.

\begin{figure}[h]
\[ \scalebox{0.8}{
    \centering
			\begin{tikzpicture}[line cap=round,line join=round ,x=2.0cm,y=1.8cm]
				\clip(-2.2,-0.1) rectangle (4.1,2.5);
					\draw [->] (-0.8,0.2) -- (-0.2,0.8);
					\draw [->] (1.2,0.2) -- (1.8,0.8);
					\draw [->] (3.2,0.2) -- (3.8,0.8);
					\draw [->] (-1.8,1.2) -- (-1.2,1.8);
					\draw [->] (0.2,1.2) -- (0.8,1.8);
					\draw [->] (2.2,1.2) -- (2.8,1.8);
					\draw [dashed] (-0.8,0.0) -- (0.8,0.0);
					\draw [dashed] (1.2,0.0) -- (2.8,0.0);
					\draw [dashed] (-1.8,1.0) -- (-0.2,1.0);
					\draw [dashed] (0.2,1.0) -- (1.8,1.0);
					\draw [dashed] (2.2,1.0) -- (3.8,1.0);
					\draw [dashed] (-2.0,0.0) -- (-1.2,0.0);
					\draw [dashed] (3.2,0.0) -- (4.0,0.0);
					\draw [->] (0.2,0.8) -- (0.8,0.2);
					\draw [->] (2.2,0.8) -- (2.8,0.2);
					\draw [->] (-1.8,0.8) -- (-1.2,0.2);
					\draw [->] (-0.8,1.8) -- (-0.2,1.2);
					\draw [->] (1.2,1.8) -- (1.8,1.2);
					\draw [->] (3.2,1.8) -- (3.8,1.2);
				
				\begin{scriptsize}
					\draw[color=black] (-1,0) node {$\rep{2}$};
					\draw[color=black] (1,0) node {$\rep{1}$};
					\draw[color=black] (3,0) node {$\rep{3}$};
					\draw[color=black] (-2,1) node {$\rep{2\\3}$};
					\draw[color=black] (0,1) node {$\rep{1\\2}$};
					\draw[color=black] (2,1) node {$\rep{3\\1}$};
					\draw[color=black] (4,1) node {$\rep{2\\3}$};
					\draw[color=black] (-1,2) node {$\rep{1\\2\\3}$};
					\draw[color=black] (1,2) node {$\rep{3\\1\\2}$};
					\draw[color=black] (3,2) node {$\rep{2\\3\\1}$};
				\end{scriptsize}
			\end{tikzpicture}
			}\]
\caption{The Auslander-Reiten quiver of $\mathrm{mod}(A)$.}
    \label{ARquiver}
   \end{figure}
   
Now, let us consider the following $\tau$-tilting $A$-module $M=\rep{1\\2\\3}\oplus \rep{1\\2}\oplus 2$  (we refer to \cite{AIR} for the basic theory of $\tau$-tilting theory). In this case, we can compute $\mathrm{Fac}(M)$, that is, the category of all factor modules of finite direct sums of copies of $M$.
We can see that the indecomposable objects of $\mathrm{Fac}(M)$  are given by the following set $\left\{\rep{1\\2\\3}, \rep{1\\2}, 2, 1\right\}$. By \cite[Corollary 2.8]{AIR}, we have that $\mathrm{Fac}(M)$ is a functorially finite torsion class. We can see that the indecomposable objects in the corresponding torsion free class are given by  $\mathcal{F}=M^{\perp}=\{\rep{2\\3}, 3\}$.  By a result due to Auslander-Smal\o-Hoshino (see for example \cite[Proposition 1.2]{AIR}), we have that $\mathcal{F}$ is a functorially finite torsion free class.\\
By work of Beligiannis (see \cite[Theorem 3.1]{BeligianisMarmaradis}), we have that $\mathrm{mod}(\Lambda)/\mathcal{I}_{\mathcal{F}}=\underline{\mathrm{mod}}_{\mathcal{F}}(\Lambda)$  has a natural structure of left triangulated category. Moreover, by Corollary \ref{lonseqiencetor}, we can relate the Hochschild-Mitchell cohomology of $\mathrm{mod}(\Lambda)$ and $\mathrm{mod}(\Lambda)/\mathcal{I}_{F}$. The Auslander-Reiten quiver of $\mathrm{mod}(A)/\mathcal{I}_{\mathcal{F}}$ is given in the figure \ref{estable}.

   \begin{figure}[h]
\[ \scalebox{0.8}{
    \centering
			\begin{tikzpicture}[line cap=round,line join=round ,x=2.0cm,y=1.8cm]
				\clip(-2.2,-0.1) rectangle (4.1,2.5);
					\draw [->] (-0.8,0.2) -- (-0.2,0.8);
					\draw [->] (1.2,0.2) -- (1.8,0.8);
					\draw [->] (0.2,1.2) -- (0.8,1.8);
					\draw [->] (2.2,1.2) -- (2.8,1.8);
					\draw [dashed] (-0.8,0.0) -- (0.8,0.0);
					\draw [dashed] (0.2,1.0) -- (1.8,1.0);
					\draw [->] (0.2,0.8) -- (0.8,0.2);
					\draw [->] (-0.8,1.8) -- (-0.2,1.2);
					\draw [->] (1.2,1.8) -- (1.8,1.2);
				
				\begin{scriptsize}
					\draw[color=black] (-1,0) node {$\rep{2}$};
					\draw[color=black] (1,0) node {$\rep{1}$};
					\draw[color=black] (0,1) node {$\rep{1\\2}$};
					\draw[color=black] (2,1) node {$\rep{3\\1}$};
					\draw[color=black] (-1,2) node {$\rep{1\\2\\3}$};
					\draw[color=black] (1,2) node {$\rep{3\\1\\2}$};
					\draw[color=black] (3,2) node {$\rep{2\\3\\1}$};
				\end{scriptsize}
			\end{tikzpicture}
			}\]
\caption{The Auslander-Reiten quiver of $\mathrm{mod}(A)/\mathcal{I}_{\mathcal{F}}$}
    \label{estable}
   \end{figure}

Finally, we will see that certain recollement of abelian categories can be lifted to a recollement of derived categories. We refer to \cite{Psaro3},  for the basic notions of recollements in abelian and triangulated categories.

\begin{theorem}\label{recolleabel}
Let $\mathcal{B}$ be a full additive subcategory of $\mathcal{C}$ and suppose that $\mathcal{I}_{\mathcal{B}}$ is a strongly idempotent ideal. Consider the recollement of abelian categories

$$\xymatrix{\mathrm{Mod}(\mathcal{C}/\mathcal{I}_{\mathcal{B}})\ar[rr]|{\pi_{\ast}=\pi_{!}}  &  & \mathrm{Mod}(\mathcal{C})\ar[rr]|{i^{!}=i^{\ast}}\ar@<-2ex>[ll]_{\pi^{\ast}}\ar@<2ex>[ll]^{\pi^{!}}  &   & \mathrm{Mod}(\mathcal{B})\ar@<-2ex>[ll]_{i_{!}}\ar@<2ex>[ll]^{i_{\ast}}}$$
given in \cite[Theorem 3.6]{LGOS2}. Then we have a recollement of triangulated categories

$$\xymatrix{\mathrm{D}(\mathrm{Mod}(\mathcal{C}/\mathcal{I}_{\mathcal{B}}))\ar[rr]|{\pi_{\ast}=\pi_{!}}  &  & \mathrm{D}(\mathrm{Mod}(\mathcal{C}))\ar[rr]|{i^{!}=i^{\ast}}\ar@<-2ex>[ll]_{L(\pi^{\ast})}\ar@<2ex>[ll]^{R(\pi^{!})}  &   & \mathrm{D}(\mathrm{Mod}(\mathcal{B})).\ar@<-2ex>[ll]_{L(i_{!})}\ar@<2ex>[ll]^{R(i_{\ast})}}$$


\end{theorem}
\begin{proof}
Consider the functor $\pi:\mathcal{C}\longrightarrow \mathcal{C}/\mathcal{I}_{\mathcal{B}}$ and the inclusion $i:\mathcal{B}\longrightarrow \mathcal{C}$.
We have the following exact sequence of abelian categories
$$\xymatrix{0\ar[r] & \mathrm{Mod}(\mathcal{C}/\mathcal{I}_{\mathcal{B}})\ar[r]^{\pi_{\ast}}& \mathrm{Mod}(\mathcal{C})\ar[r]^{i^{\ast}} & \mathrm{Mod}(\mathcal{B})\ar[r] & 0.}$$
By \cite[Theorem 3.2]{Miyachi}, we get the following exact sequence of triangulated categories
$$\xymatrix{0\ar[r] & \mathrm{D}(\mathrm{Mod}(\mathcal{C}))_{\pi_{\ast}\big(\mathrm{Mod}(\mathcal{C}/\mathcal{I}_{\mathcal{B}})\big)}\ar[r] & \mathrm{D}(\mathrm{Mod}(\mathcal{C}))\ar[r]^{i^{\ast}} & \mathrm{D}(\mathrm{Mod}(\mathcal{B}))\ar[r] & 0}$$
where $\mathrm{D}(\mathrm{Mod}(\mathcal{C}))_{\pi_{\ast}\big(\mathrm{Mod}(\mathcal{C}/\mathcal{I}_{\mathcal{B}})\big)}$ is the full subcategory of $\mathrm{D}(\mathrm{Mod}(\mathcal{C}))$ consisting of the complexes whose homology belongs to $\pi_{\ast}\Big(\mathrm{Mod}(\mathcal{C}/\mathcal{I}_{\mathcal{B}})\Big)\simeq  \mathrm{Mod}(\mathcal{C}/\mathcal{I}_{\mathcal{B}})$.\\
Consider the inclusion $j:\pi_{\ast}\big(\mathrm{Mod}(\mathcal{C}/\mathcal{I}_{\mathcal{B}})\big)\longrightarrow \mathrm{Mod}(\mathcal{C})$, we have that $\pi_{\ast}\big(\mathrm{Mod}(\mathcal{C}/\mathcal{I}_{\mathcal{B}})\big)$ is a Serre subcategory of $\mathrm{Mod}(\mathcal{C})$.\\ 
Now, we consider the induced functor $j_{\ast}:\mathrm{D}\big(\pi_{\ast}\big(\mathrm{Mod}(\mathcal{C}/\mathcal{I}_{\mathcal{B}})\big)\big)\longrightarrow 
\mathrm{D}\big(\mathrm{Mod}(\mathcal{C})\big)$. By \cite[Proposition 1.7.11]{Kashiwara} in page 49, we get that $j_{\ast}$ gives an equivalence
$$\mathrm{D}\big(\pi_{\ast}\big(\mathrm{Mod}(\mathcal{C}/\mathcal{I}_{\mathcal{B}})\big)\big)\simeq 
\mathrm{D}\big(\mathrm{Mod}(\mathcal{C})\big)_{\pi_{\ast}\big(\mathrm{Mod}(\mathcal{C}/\mathcal{I}_{\mathcal{B}})\big)}.$$
Hence, we obtain the following exact sequence of triangulated categories

\begin{equation}\label{Sucexacta tring}
\xymatrix{0\ar[r] & \mathrm{D}(\mathrm{Mod}(\mathcal{C}/\mathcal{I}_{\mathcal{B}}))\ar[r]^{\pi_{\ast}} & \mathrm{D}(\mathrm{Mod}(\mathcal{C}))\ar[r]^{i^{\ast}} & \mathrm{D}(\mathrm{Mod}(\mathcal{B}))\ar[r] & 0}
\end{equation}
Now, we consider the following diagram of adjoint functors
$$\xymatrix{\mathrm{Mod}(\mathcal{C}/\mathcal{I}_{\mathcal{B}})\ar[rr]|{\pi_{\ast}=\pi_{!}}  &  & \mathrm{Mod}(\mathcal{C}).\ar@<-2ex>[ll]_{\pi^{\ast}}\ar@<2ex>[ll]^{\pi^{!}} }$$
Since $\mathrm{Mod}(\mathcal{C})$ and  $\mathrm{Mod}(\mathcal{C}/\mathcal{I}_{\mathcal{B}})$ are abelian categories with enough injectives and projectives, this adjunction pass to derived categories
$$\xymatrix{\mathrm{D}(\mathrm{Mod}(\mathcal{C}/\mathcal{I}_{\mathcal{B}}))\ar[rr]|{\pi_{\ast}=\pi_{!}}  &  &\mathrm{D}(\mathrm{Mod}(\mathcal{C})),\ar@<-2ex>[ll]_{L(\pi^{\ast})}\ar@<2ex>[ll]^{R(\pi^{!})} }$$
where $L(\pi^{\ast})$ and $R(\pi^{!})$ denote the left and right derived functors of $\pi^{\ast}$ and $\pi^{!}$ respectively. Since $\pi:\mathcal{C}\longrightarrow \mathcal{C}/\mathcal{I}$ is a homological epimorphism, we have that $\pi_{\ast}$ is full and faithful. By \cite[Theorem 2.1]{PS1}, we can complete the previous diagram to a recollement where the third category is the Verdier quotient of $\mathrm{D}(\mathrm{Mod}(\mathcal{C}))$ by $\mathrm{D}(\mathrm{Mod}(\mathcal{C}/\mathcal{I}))$.
By the sequence given in Equation \ref{Sucexacta tring}, we have that 
 $\frac{\mathrm{D}(\mathrm{Mod}(\mathcal{C}))}{\mathrm{D}(\mathrm{Mod}(\mathcal{C}/\mathcal{I}_{\mathcal{B}}))}\simeq \mathrm{D}(\mathrm{Mod}(\mathcal{B}))$.\\
 Hence, by  \cite[Theorem 2.1]{PS1}, we have a recollement 
 $$\xymatrix{\mathrm{D}(\mathrm{Mod}(\mathcal{C}/\mathcal{I}_{\mathcal{B}}))\ar[rr]|{\pi_{\ast}=\pi_{!}}  &  & \mathrm{D}(\mathrm{Mod}(\mathcal{C}))\ar[rr]|{i^{!}=i^{\ast}}\ar@<-2ex>[ll]_{L(\pi^{\ast})}\ar@<2ex>[ll]^{R(\pi^{!})}  &   & \mathrm{D}(\mathrm{Mod}(\mathcal{B})).\ar@<-2ex>[ll]_{F}\ar@<2ex>[ll]^{G}}$$
Similarly, by considering the following diagram of adjoint functors
$$\xymatrix{\mathrm{Mod}(\mathcal{C})\ar[rr]|{i^{!}=i^{\ast}}  &   & \mathrm{Mod}(\mathcal{B}),\ar@<-2ex>[ll]_{i_{!}}\ar@<2ex>[ll]^{i_{\ast}}}$$
we obtain the following diagrama of adjoint functors
$$\xymatrix{\mathrm{D}(\mathrm{Mod}(\mathcal{C}))\ar[rr]|{i^{!}=i^{\ast}}  &   & \mathrm{D}(\mathrm{Mod}(\mathcal{B})).\ar@<-2ex>[ll]_{L(i_{!})}\ar@<2ex>[ll]^{R(i_{\ast})}}$$
We conclude that $F=L(i_{!})$ and $G=R(i_{\ast})$, proving the result.
\end{proof}

\section*{Acknowledgements}
The authors thank Professor Eduardo do Nascimento Marcos for sharing his ideas related to this paper's content with us and for the many helpful discussions.\\

\bibliographystyle{amsplain}

\begin{thebibliography}{ABPRS}






















\bibitem{AIR} T. Adachi,  O. Iyama,  I. Reiten. $\tau $-{\it{tilting theory.}} Compositio Mathematica. 2014, 150 (3): 415-452.

\bibitem{APG} M. Auslander, M.I. Platzeck and G. Todorov. {\it{Homological theory of idempotent ideals}}. Trans. Amer. Math. Soc.  Vol. 332, No. 2, pp. 667-692 (1992).

\bibitem{AusBook} M. Auslander, I. Reiten, S. Smal\o. {\it{Representation theory of artin algebras.}} Studies in Advanced Mathematics 36, Cambridge University Press (1995).

\bibitem{BeligianisMarmaradis} A. Beligiannis and N. Marmaridis. {\it{ Left triangulated categories arising from contravariantly finite subcategories}}.
Commun.  Algebra  22 (12), 5021-5036 (1994).

\bibitem{Cartan} H. Cartan, S. Eilenberg. {\it{Homological Algebra}}. Princeton University Press, Princeton, New Jersey, (1956).

\bibitem{Chites} A. Chites, C. Chites. {\it{Separable K-linear categories}}, Cent. Eur. J. Math. 8(2), 2010, 274-281
DOI: 10.2478/s11533-010-0007-6.

\bibitem{Cibils} C. Cibils. Tensor Hochschild homology and cohomology.  Interactions between
ring theory and representations of algebras (Murcia), 35-51, Lecture Notes in Pure and Appl. Math., 210, Dekker, New York, (2000).

\bibitem{CibilsMarcos} C. Cibils, E. N. Marcos. {\it{Resolving by a free action linear category and applications to Hochschild-Mitchell (co)homology}}. J. Algebra 59, 117-141 (2022).

\bibitem{CibilsRedondo} C. Cibils, M. J. Redondo. {\it{Cartan-Leray spectral sequence for Galois coverings of
categories}} J. Algebra 284, 310-325 (2005).

\bibitem{CibiSoloRedon} C. Cibils, M. J. Redondo, A. Solotar. {\it{The intrinsic fundamental group of a linear category}}.  Algebr. Represent. Theor. 15 (4), pp. 735-753, (2012).


\bibitem{PS1} E. Cline, B. Parshall, L. Scott. {\it{Algebraic stratification in representation categories}}. Journal of Algebra. Volume 117, Issue 2,  pp. 504-521 (1988).

\bibitem{GeigleLen} W. Geigle, H. Lenzing. {\it{ Perpendicular categories with applications to representations and sheaves}}. J. Algebra 144: 273-343, (1991).

\bibitem{EduardoGreen} E. L. Green, E. N. Marcos,  N. Snashall. {\it{The Hochschild Cohomology Ring of a One Point Extension.}} Commun. Algebra, Vol. 31, Issue 1, 357-379 (2003).

\bibitem{Happel} D. Happel. {\it{Hochschild cohomology of finite dimensional algebras}}, in `S\'eminaire M. P. Malliavin, Paris, 1987-1988,'' Lecture Notes in Mathematics, Vol. 1404, pp. 108-126, Springer-Verlag, New York , Berlin (1989).

\bibitem{HersSolotar1} E. Herscovich, A. Solotar. {\it{Hochschild-Mitchell cohomology and Galois extensions}}. J. Pure Appl. Algebra 209 (1), pp. 37-55, (2007).

\bibitem{HersSolotar2} E. Herscovich, A. Solotar. {\it{Derived invariance of Hochschild-Mitchell (co)homology and one-point extensions. }} J. Algebra 315 (2), pp. 852-873, (2007).

\bibitem{Hochschild} G. Hochschild. {\it{On the cohomology groups of an associative algebra}}.  Ann. of Math. 46, 58-67 (1946).

\bibitem{Kashiwara}  M. Kashiwara, P. Shapira. {\it{Sheaves on Manifolds}}.
Springer- Verlag, Berling-Heidelberg, New York (1990).

\bibitem{KoenigNagase} S. Koenig,  H. Nagase. {\it{Hochschild cohomology and stratifying ideals}}. J. Pure
Appl. Algebra 213, 886-891 (2009).

\bibitem{Heninepi} H. Krause. {\it{Epimorphisms of additive categories up to direct factors.}} J. Pure Appl. Algebra  (203), 2005, 113-118.

\bibitem{LeOS} A. Le\'on-Galeana, M. Ort\'iz-Morales, V. Santiago. {\it{Triangular Matrix Categories I: Dualizing Varieties and generalized one-point extension}}. Algebr. Represent. Theor. Published on line: 22 February 2022, 1-50.
https://doi.org/10.1007/s10468-022-10114.

\bibitem{LGOS2}  A. Le\'on-Galeana, M. Ort\'iz-Morales, V. Santiago. {\it{Triangular matrix categories II: Recollements and functorially finite subcategories}}. Algebr Represent Theor 26, 783-829 (2023). https://doi.org/10.1007/s10468-022-10113-w

\bibitem{MichelanaPlat} S. Michelena, M. I. Platzeck. {\it{Hochschild cohomology of triangular matrix algebras}}. J.
Algebra 233, 502-525 (2000).

\bibitem{MitBook} B. Mitchell. {\it{Theory of categories.}} Pure and Appl. Math., Vol 17,  Academic Press, New York  (1965).

\bibitem{Mitchelring} B. Mitchell. {\it{Rings with several objects.}}
Adv. in Math.  Vol. 8, 1-161 (1972).

\bibitem{Miyachi} J. Miyachi. {\it{Localization of triangulated categories and derived categories}} J. Algebra, 141 (1991),463-483.

\bibitem{Psaro3}  C. Psaroudakis. {\it{A representation-theoretic approach to recollements of abelian categories.}} Contemp. Math. of Amer. Math. Soc. Vol. 716, 67-154 (2018).

\bibitem{Redondo} M. J. Redondo. {\it{Hochschild cohomology: some methods for computations}}. Resenhas IME-USP, Vol. 5, No. 2, 113-137, IX Algebra Meeting USP/UNICAMP/UNESP  (2001).

\bibitem{RSS} L.G Rodr\'iguez-Vald\'es, M. L. S Sandoval-Miranda, V. Santiago-Vargas. {\it{Homological theory of k-idempotent ideals in dualizing varieties.}} Commun.  Algebra (2021) 

\end{thebibliography}

\end{document}